\documentclass[pdflatex,sn-mathphys-num,leqno]{sn-jnl} 
\usepackage{tikz}
\usetikzlibrary{trees, arrows.meta, positioning}
\usepackage{caption}
\usepackage{subcaption}

\usepackage{graphicx}%
\usepackage{multirow}%
\usepackage{amsmath,amssymb,amsfonts}%
\usepackage{amsthm}%
\usepackage{mathrsfs}%
\usepackage[title]{appendix}%
\usepackage{xcolor}%
\usepackage{textcomp}%
\usepackage{manyfoot}%
\usepackage{booktabs}%
\usepackage{algorithm}%
\usepackage{algorithmicx}%
\usepackage{algpseudocode}%
\usepackage{listings}%

\usepackage{bm}
\usepackage{mathdots}
\usepackage{pifont}
\newcommand{\cmark}{\ding{51}}%
\newcommand{\xmark}{\ding{55}}%
\usepackage{pdflscape}

\DeclareMathOperator*{\argmin}{argmin}

\DeclareMathOperator{\conv}{conv}
\newcommand{\R}{\mathbb{R}}
\newcommand{\vol}[2]{\mathrm{vol}_{\scriptscriptstyle #1}(#2)}

\makeatletter
\renewcommand*{\top}{%
  {\mathpalette\@transpose{}}%
}
\newcommand*{\@transpose}[2]{%
  \raisebox{\depth}{$\m@th#1\scriptscriptstyle\mathsf{T}$}%
}
\makeatother

\theoremstyle{thmstyleone}%
\newtheorem{theorem}{Theorem}%

\theoremstyle{thmstyletwo}%

\theoremstyle{thmstylethree}%

\begin{document}

\title[The convex hull of a simple monomial]{On the convex hull of\break the graph of a simple monomial\footnote{This manuscript is an extended version of the very short conference paper \cite{LSS_STOGO}}}

\author[1]{\fnm{Jon} \sur{Lee}}\email{jonxlee@umich.edu}

\author[2]{\fnm{Daphne} \sur{Skipper}}\email{skipper@usna.edu}

\author[3]{\fnm{Emily} \sur{Speakman}}\email{emily.speakman@ucdenver.edu}

\affil[1]{\orgdiv{IOE Department}, \orgname{University of Michigan}, \orgaddress{\city{Ann Arbor}}}

\affil[2]{\orgdiv{Mathematics Department}, \orgname{United States Naval Academy}, \orgaddress{\city{Annapolis}}}

\affil[3]{\orgdiv{Mathematical and Statistical Sciences Department}, \orgname{University of Colorado}, \orgaddress{\city{Denver}}}

\abstract{Motivated by previous efforts toward mathematically analyzing the treatment of
monomials in spatial branch-and-bound, we study the convex hull of the graph of
a simple monomial on a nonnegative box domain in arbitrary dimension, 
where at most one of the variable lower bounds is positive. 
We give: (i) a description via linear inequalities, and (ii) a formula for the
volume.}


\maketitle

\section{Introduction}\label{sec:intro}

A common approach for global optimization 
is spatial branch-and-bound (sB\&B), see e.g. \cite{SMITH1999457}, which is the back-bone algorithm 
in several successful packages, for example Baron\footnote{\url{https://www.minlp.com/baron--solver}}, SCIP\footnote{\url{https://www.scipopt.org/}}, and COUENNE\footnote{\url{https://github.com/coin-or/Couenne}}.
These packages aim specifically at factorable formulations, where each function is a low-depth composition of library functions of a small number of 
variables (or of very-special structure in many variables) and affine functions (in an arbitrary number of variables). It is assumed that with each variable 
we have an interval domain, and that for each library function, we have a practical 
convexification method for the graph 
of that function on box subdomains of the function domain. 
In practice, we may have good convexification methods
for only some box subdomains for a given
library function. 

There is great modeling interest and power in 
polynomials, which at their heart, from the point of 
view of sB\&B boil down to monomials. 
Noting that we can decompose monomials into lower degree
monomials, and considering that we lose something
in the convexifications as we decompose, 
it behooves us to increase the 
set of library monomials 
and associated box subdomains at our disposal,
and to understand the quality of the convexifications that 
we employ.

Toward this, in what follows, we study 
the convexification of graphs of simple monomials 
on special nonnegative box subdomains.
We aim at getting explicit convex-hull descriptions 
with minimal inequality systems, and 
we aim at calculating the volume (i.e., Lebesgue measure)
of these convex hulls, toward future work on
evaluating the quality of relaxations of these convex hulls.

We assume some familiarity with polytopes (see \cite{zieg}, for example).  Briefly, $\conv(\cdot)$ denotes convex hull, and
$\vol{d}{B}$ denotes the $d$-dimensional volume (i.e., Lebesgue measure) of a convex body $B\subseteq \R^d$.

For $n\geq 2$ and  $\bm{a},\bm{b}\in\R^n$, with
$0\leq a_i <b_i$\,, $i=1,2,\ldots,n$, let
\[ \textstyle
C_n(\bm{a},\bm{b}):=\conv \left (\left\{\begin{pmatrix} \bm{x} \\ y \end{pmatrix} \in \R^{n+1} ~:~
\bm{a} \leq \bm{x} \leq \bm{b},~ y := \prod_{i=1}^n x_i\right\}\right ).
\]
It is well known that this fundamental object, the convex hull of the graph of the ($n$-linear) monomial
$\prod_{i=1}^n x_i$ on the box domain $[\bm{a},\bm{b}]$,
is a polytope having $2^n$ extreme points, each of which has the form
$\left(
\bm{x}^\top, y \right)^\top$, where each of the $n$ components of $\bm{x}$ is set to either its lower or upper bound, and $y:=x_1x_2 \dots x_n$\,; see \cite{rikun1997}.
For conciseness,  we let $C_n^k$ denote $C_n(\bm{a},\bm{b})$ in the case of $a_i:=0$, for $i=1,2,\ldots,n-k$.  That is, we allow at most $k$ variables (arbitrarily $x_{n-k+1}, \ldots, x_n$) to have a non-zero lower bound. 

The case of $k=0$ is quite simple, see e.g. \cite[Sec. 5.1]{LeeSkipperSpeakmanMPB2018}.
In that case it is not hard to see that 
$C_n^0$ is a pyramid with bases that is the intersection
of $C_n^0$ with the hyperplane described by $y=0$.
The extreme points of the base are the $2^n-1$ points
of the form $(\bm{x}^\top,y)^\top$, where $\bm{x}\in\mathbb{R}^n$
has \emph{fewer} than $n$ components at their upper bounds 
and the remaining at 0. The base is the solution set of
\begin{align*}
& 0\leq x_i \leq b_i \,,\quad i=1,\ldots,n;\\
& \sum_{i=1}^n \left( \prod_{j\not= i} b_j\right) x_i \leq (n-1)\prod_{i=1}^n b_i\,,
\end{align*}
the last inequality of which cuts off $(b_1,\ldots,b_n)^\top$
at the adjacent $n$ vertices of the box described 
by the other inequalities. To see now what 
the facets of $C_n^0$ look like (not explicitly done in \cite{LeeSkipper2017}), we just need to
``lift'' the variable $y\geq 0$ into these inequalities,
so that the resulting inequality is satisfied as an equation
by the pyramid apex $(b_1\,,\ldots,b_n\,,\prod_{i=1}^n b_i)^\top$. It is easy to check that the last inequality 
lifts to 
\begin{align*}
& -y + \sum_{i=1}^n \left( \prod_{j\not= i} b_j\right) x_i \leq (n-1)\prod_{i=1}^n b_i\,.
\end{align*}
The inequalities $x_i\geq 0$ ($i=1,\ldots,n$)
lift to 
\[
-\frac{1}{\prod_{j\not=i} b_i}y + x_i \geq 0, \quad i=1,\ldots,n,
\]
and the lifting coefficient of $y$ for the
$x_i \leq b_i$ inequalities are all 0  ($i=1,\ldots,n$).

Using the well-known volume formula for a pyramid,
\cite[Thm. 23]{LeeSkipperSpeakmanMPB2018} obtained
\[
\vol{n+1}{C^0_n}=\frac{\prod_{i=1}^n b_i^2}{(n+1)!}(n!-1).
\]

In much of what follows, we concentrate on 
the significantly more challenging case of $C_n^1$\,, which has at most one variable (arbitrarily $x_n$) having a nonzero lower bound. This significantly generalizes what we just reviewed for 
the pyramid $C_n^0$\,.
In Section \ref{sec:hull}, we establish a complete linear-inequality description of 
$C_n^1$\,, for all $n$. In Section \ref{sec:vol},
we establish a formula for the volume of 
$C_n^1$\,, for all $n$.

\medskip
\noindent
{\bf Some closely related literature.} 
It is very easy to check that for the bilinear case, 
 $C_2^2$ is the 
  tetrahedron
 described by the well-known McCormick inequalities: 
\begin{align*}
    y-a_2x_1-a_1x_2+a_1 a_2 \geq 0;&\\
    -y+b_2x_1+a_1x_2-a_1b_2 \geq 0;&\\
    -y+a_2x_1+b_1x_2-b_1a_2 \geq 0;&\\
    y-b_2x_1-b_1 x_2 +b_1b_2 \geq 0,&
\end{align*}
with volume
\[
\vol{3}{C^2_2}=\frac{1}{6}(b_1-a_1)^2(b_2-a_2)^2;
\]
see \cite{McCormick76} and \cite[page 133]{LeeSkipperSpeakmanMPB2018}. 
The quite complicated trilinear case of $C_3^3$ was studied extensively. \cite{meyer2004trilinear, meyer04mixed} worked out the
facet-describing inequalities, even for the 
case of mixed-sign (box) domains.
\cite{SpeakmanThesis,SpeakmanLee2015}
found the volume formula
(and, for various natural polyhedral relaxations of $C_3^3$\,, they worked out 
the
facet-describing inequalities and volume formulae). \cite{SpeakAkerov2019}
devised a completely different proof of the volume formula for $C_3^3$ using ``mixed volumes''. 
\cite{MakhoulSpeakman2025} generalized the 
volume formula for $C_3^3$ to the 
case of mixed-sign (box) domains.
Additional results concerning $C_3^3$\,, 
motivated more directly by its use in spatial branch-and-bound are in 
\cite{SpeakmanLee_Branching} and \cite{SpeakmanYuLee}.
\cite{quad}, motivated
by problems in molecular distance geometry
and the Hartree–Fock problem in physics,
made a mostly-computational study,
comparing relaxations of $C_4^4$\,, considering even mixed-sign domains.

\medskip
\noindent
{\bf Broader background literature on volumes.} 
\cite{LM1994} introduced the idea of comparing
polytopes arising in global and combinatorial optimization
by obtaining and analyzing volume formulae.
Further work in this direction includes
\cite{KLS1197} and \cite{Steingrimsson1994}.
A fairly recent and broader survey  is \cite{LeeSkipperSpeakmanMPB2018}, which includes further 
references.

\medskip
\noindent
{\bf Standard terminology for polytopes.} 
The inequality $\alpha^\top  x \leq \beta$ is 
\emph{valid} for a polytope $\mathcal{P}\subseteq\mathbb{R}^d$ if it is satisfied by
all points of $\mathcal{P}$. 
The \emph{face described} by
a valid inequality $\alpha^\top  x \leq \beta$  of $\mathcal{P}$ 
is $\mathcal{P}\cap \{ x \in \mathbb{R}^d ~:~ \alpha^\top x = \beta\}$.
The \emph{dimension} of $\mathcal{P}$, $\dim(\mathcal{P})$, is the maximum number of affinely-independent points in $\mathcal{P}$ minus one.
The polytope $\mathcal{P}\subseteq\mathbb{R}^d$ is
\emph{full dimensional} if $\dim(\mathcal{P})=d$.
A face $\mathcal{F}$ of polytope $\mathcal{P}$ is a
\emph{facet} of $\mathcal{P}$ 
if $\dim(\mathcal{F})=\dim(\mathcal{P})-1$.
If polytope $\mathcal{P}$ is
full-dimensional, then, up to positive scaling, every
facet of $\mathcal{P}$ is described by a unique 
inequality, and all such inequalities 
yield a minimal inequality description of $\mathcal{P}$.

\medskip
\noindent
{\bf Additional notation.} 
For a positive integer $n$, we let $N:=\{1,\ldots,n\}$.
For a list of (unique) integers $X$ 
from a set of integers $S$, we let $S_X$ denote
the set of integers in $S$ that are not on $X$. For example, for $i\in N$,  $N_{in}=\{1,\ldots,i-1,i+1,\ldots,n-1\}$. For any finite set $S$ and element $e$ not in (resp., in) $S$, we let $S+e:=S\cup\{e\}$ (resp., $S-e:=S\setminus\{e\})$.


\section{Convex hull}\label{sec:hull}

In this section, we establish a complete linear-inequality description of 
$C_n^1$\,, for all $n$. While our methodology for doing so is well known, carrying 
it out is quite complicated, due to degeneracy. 

\begin{theorem}\label{thm:facets}
For all $n\geq 2$,  
$C_n^1$ is completely described by the system
\begin{align*}
y -   \sum_{i\in N_n} \left ( a_n \prod_{j\in N_{in}}b_j \right ) x_i + (n-2)a_n \prod_{j\in N_n}b_j &\geq 0; &\\
y - \sum_{i\in N} \left (\prod_{j\in N_i}b_j \right ) x_i + (n-1) \prod_{j\in N}b_j &\geq 0; &\\
-y + \left ( a_n \prod_{j\in N_{in}}b_j \right ) x_i + \left (\prod_{j\in N_n} b_j \right ) x_n - a_n \prod_{j\in N_n}b_j &\geq 0, &i\in N_n\,;\\
-y + \left (\prod_{j\in N_{i}} b_j \right )x_i &\geq 0, &i\in N_n\,;\\
y &\geq 0; & \\
x_i &\leq b_i\,, &i\in N; \\
x_n &\geq a_n\,. & 
\end{align*}
\end{theorem}

\newpage

\begin{proof}~
\paragraph{Outline} We first verify that every inequality in the system is valid for $C^1_n$\,.
We then formulate a linear program~(P) that maximizes an arbitrary linear function
$c_0 y + \sum_{i \in N} c_i x_i$ over the inequality system, and give an algorithm
that identifies an optimal primal solution among the extreme points of~$C^1_n$\,.
For each possible outcome of this algorithm, we construct an explicit dual feasible
solution whose objective value matches the primal objective, establishing optimality
by strong duality. Because validity ensures that every extreme point of~$C^1_n$ is
feasible for~(P), and optimality demonstrates that no feasible point of~(P) achieves a
higher objective value than some extreme point of~$C^1_n$\,, we may conclude that the feasible region
of~(P) is exactly~$C^1_n$\,.


\paragraph{System Validity}

To establish that the inequality system in Theorem~\ref{thm:facets} is valid for $C^1_n$\,, we evaluate each inequality at every extreme point of $C^1_n$\,.  For convenience, we first define
$$\Pi  := \prod_{i \in N_n} b_i\,,$$ noting that $\Pi > 0$ because by assumption $b_1, \dots ,b_{n-1} >0$.  

At each extreme point of~$C^1_n$\,, every $x_i$ ($i \in N_n$) is set to either~$0$ or~$b_i$\,, and $x_n \in \{a_n, b_n\}$. Letting $T \subseteq N_n$ denote the set of indices with $x_i = b_i$ (so that $x_i = 0$ for $i \in N_n \setminus T$), we have $y = x_n\Pi$ when $T = N_n$ and $y = 0$ when $T \subsetneq N_n$\,. Table~\ref{tab:validity} records the resulting left-hand side value for each inequality at each extreme point type.  

Most entries in the table are immediately nonnegative given our assumptions on the problem data, with the following two observations completing the verification.  First, when $T \subsetneq N_n$ we have $|T| \leq n-2$, so the factor $(n-2-|T|)$ appearing in several entries is nonnegative.  Second, the entry for the second inequality with $T \subsetneq N_n$ and $x_n = a_n$ satisfies $$(n-1-|T|)\,b_n\Pi - a_n\Pi \geq b_n\Pi - a_n\Pi = (b_n - a_n)\Pi \geq 0.$$

\begin{table}[ht]
\centering
\caption{Left-hand side values of each inequality in the system of Theorem~\ref{thm:facets} evaluated at every extreme point of~$C^1_n$\,.}
\label{tab:validity}
\medskip
\renewcommand{\arraystretch}{1.8}
\begin{tabular}{@{}lcccc@{}}
\toprule
& \multicolumn{2}{c}{$T = N_n$ \;($y = x_n\Pi$)}
& \multicolumn{2}{c}{$T \subsetneq N_n$ \;($y = 0$, \;$|T| \leq n\!-\!2$)} \\
\cmidrule(lr){2-3} \cmidrule(lr){4-5}
Inequality & $x_n = b_n$ & $x_n = a_n$ & $x_n = b_n$ & $x_n = a_n$ \\
\midrule
$\displaystyle y - \sum_{i\in N_n} \frac{a_n}{b_i}\,\Pi\, x_i + (n\!-\!2)\,a_n\Pi \geq 0$
  & $(b_n\!-\!a_n)\Pi$
  & $0$
  & $(n\!-\!2\!-\!|T|)\,a_n\Pi$
  & $(n\!-\!2\!-\!|T|)\,a_n\Pi$ \\[4pt]
$\displaystyle y - \sum_{i\in N_n}\!\frac{b_n}{b_i}\Pi x_i - \Pi x_n+ (n\!-\!1)\,b_n\Pi \geq 0$\rlap{$^*$}
  & $0$
  & $0$
  & $(n\!-\!2\!-\!|T|)\,b_n\Pi$
  & $(n\!-\!1\!-\!|T|)\,b_n\Pi - a_n\Pi$ \\[4pt]
\midrule
\multicolumn{5}{@{}l}{\emph{For each $i \in N_n$\,, with sub-cases
$i \in T$ and $i \notin T$: }} \\[2pt]
$\displaystyle -y + \frac{a_n}{b_i}\,\Pi\, x_i + \Pi\, x_n - a_n\Pi \geq 0$,\; $i \in T$
  & $0$
  & $0$
  & $b_n\Pi$
  & $a_n\Pi$ \\[4pt]
$\displaystyle -y + \frac{a_n}{b_i}\,\Pi\, x_i + \Pi\, x_n - a_n\Pi \geq 0$,\; $i \notin T$
  & ---\rlap{$^{**}$}
  & ---\rlap{$^{**}$}
  & $(b_n\!-\!a_n)\Pi$
  & $0$ \\[4pt]
$\displaystyle -y + \frac{b_n}{b_i}\,\Pi\, x_i \geq 0$,\; $i \in T$
  & $0$
  & $(b_n\!-\!a_n)\Pi$
  & $b_n\Pi$
  & $b_n\Pi$ \\[4pt]
$\displaystyle -y + \frac{b_n}{b_i}\,\Pi\, x_i \geq 0$,\; $i \notin T$
  & ---\rlap{$^{**}$}
  & ---\rlap{$^{**}$}
  & $0$
  & $0$ \\[4pt]
\midrule
$y \geq 0$
  & $b_n\Pi$
  & $a_n\Pi$
  & $0$
  & $0$ \\[4pt]
$x_i \leq b_i$\,, $i \in N$;\quad $x_n \geq a_n$
  & \multicolumn{4}{c}{Satisfied by construction ($x_i \in \{0, b_i\}$ for $i \in N_n$\,;\;
    $x_n \in \{a_n, b_n\}$).} \\
\bottomrule
\end{tabular}

\smallskip
{\footnotesize $^*$\, Written in an expanded form (when compared to the theorem statement) using simplifying notation.}\\
{\footnotesize $^{**}$\,$i \notin T$ cannot occur when $T = N_n$\,.}
\end{table}

\paragraph*{Primal Problem} 

We consider the following linear program with $\bm{c} := [c_0, c_1, \dots, c_n] \in \R^{n+1}$.  The feasible region is exactly the system of inequalities given in the theorem statement (using new notation).
\begin{align*}
\text{Maximize} \quad  \quad &c_0 y + \sum_{i\in N} c_i x_i& &   \tag{P}\label{primal}\\
\text{s.t.}\quad \quad &y -  \sum_{i\in N_n}\frac{a_n}{b_i} \,\Pi\, x_i \geq -(n\!-\!2)\,a_n \Pi\,;& \\
\quad\quad &y -  \sum_{i\in N_n}\frac{b_n}{b_i}\,\Pi\, x_i - \Pi\, x_n \geq -(n\!-\!1)\,b_n \Pi\,;& \\
\quad\quad &-y + \frac{a_n}{b_i}\,\Pi\, x_i + \Pi\, x_n \geq a_n \Pi\,,
&i\in N_n\,;\\
\quad\quad &-y + \frac{b_n}{b_i}\,\Pi\, x_i \geq 0\,,
&i\in N_n\,;\\
\quad\quad &y \geq 0\,;&\\
\quad\quad &x_n \geq a_n\,;&\\
\quad\quad &x_i \leq b_i\,, &i\in N\,.
\end{align*}

\medskip
Let $N_n^+ := \{i \in N_n : c_i \geq 0\}$ and $N_n^- := \{i \in N_n : c_i < 0\}$ partition $N_n$\,, and define
\begin{equation*}
S^+ := \sum_{i \in N_n^+} c_i b_i\,, \quad \text{and}\quad  S^- := \sum_{i \in N_n^-} c_i b_i\,, 
\end{equation*}
noting that $S^+ \geq 0$, $S^- \leq 0$, and $S^+ + S^- = \sum_{i \in N_n} c_i b_i$\,. 
Recall that because $a_i = 0$ for all $i \in N_n$\,, the extreme points of $C^1_n$ have a
simple structure: each $x_i$ ($i \in N_n$) is set to either~$0$ or~$b_i$\,,
$x_n$ is set to either~$a_n$ or~$b_n$\,, and $y = \prod_{i=1}^n x_i$\,.
Setting any $x_i = 0$ for $i\in N_n$ forces $y = 0$, which decouples the
remaining variables.  Therefore, every candidate optimal solution takes
one of two forms:
\begin{itemize}
\item \emph{Full product:}
  set $x_i = b_i$ for all $i \in N_n$\,, leading to $y = x_n \Pi > 0$. The only
  remaining choice is $x_n \in \{a_n, b_n\}$.

\item \emph{Zero product:}
  set at least one $x_i = 0$ for $i \in N_n$\,, forcing $y = 0$.
  Because $y = 0$, the objective reduces to $\sum_{i \in N} c_i x_i$\,,
  so each remaining $x_i$ ($i \in N_n$) is chosen independently:
  $x_i = b_i$ if $c_i \geq 0$, and $x_i = 0$ if $c_i < 0$.
  The only remaining choice is again $x_n \in \{a_n, b_n\}$.
\end{itemize}
Combining these two choices (full vs.\ zero product, and $x_n = a_n$
vs.\ $x_n = b_n$) yields exactly four candidate solutions.
The objective values of these candidates depend on whether
$N_n^- = \emptyset$ or $N_n^- \ne \emptyset$. When $N_n^- = \emptyset$, all $c_i \geq 0$ for $i \in N_n$\,,
so forcing $y = 0$ requires setting some $x_k=0$ even though $c_k \geq 0$, and thus potentially sacrificing a positive contribution to the objective value.  The best choice is to set $$k := \argmin_{j \in N_n^+} c_j b_j\,,$$
at a cost of~$c_kb_k$ (which may be zero if some $c_i=0$).   When $N_n^- \neq \emptyset$, the variables in $N_n^-$ already have $c_i < 0$,
so setting them to zero forces $y = 0$ at no cost, and in fact removes their
negative contribution from the objective.
In each case, the primal algorithm simply evaluates the correct objective value at each of the four candidate solutions and returns the best.

\paragraph{Primal Algorithm}

\medskip\noindent
\textbf{Input:} $a_n > 0$;\; $b \in \mathbb{R}^n$ with $0 < a_n < b_n$
and $b_i > 0$ for $i \in N_n$\,;\; $c_0 \in \mathbb{R}$;\; $c \in \mathbb{R}^n$.

\smallskip\noindent
\textbf{Output:} An extreme point $(\bar{x}, \bar{y})$ of~(P), which we will prove to be optimal.

\begin{itemize}
\item \textbf{Step 1: Classify.} 

Compute $N_n^+$, $N_n^-$, $S^+$, $S^-$, $\Pi$.
\medskip

\item \textbf{Step 2: Evaluate candidates.}

\medskip
\emph{Case A: $N_n^- = \emptyset$:}
\begin{alignat*}{2}
z^{\Pi}_b &:= c_0\, b_n \Pi + c_n b_n + S^+\,, &\qquad &
  \text{(full product, $x_n = b_n$)} \\
z^{\Pi}_a &:= c_0\, a_n \Pi + c_n a_n + S^+\,, &&
  \text{(full product, $x_n = a_n$)} \\
z^0_b &:= c_n b_n + S^+ - c_k b_k\,, &&
  \text{(zero product, $x_n = b_n$)} \\
z^0_a &:= c_n a_n + S^+ - c_k b_k\,. &&
  \text{(zero product, $x_n = a_n$)}
\end{alignat*}

\emph{Case B: $N_n^- \neq \emptyset$:}
\begin{alignat*}{2}
z^{\Pi}_b &:= c_0\, b_n \Pi + c_n b_n + S^+ + S^-\,, &\qquad &
  \text{(full product, $x_n = b_n$)} \\
z^{\Pi}_a &:= c_0\, a_n \Pi + c_n a_n + S^+ + S^-\,, &&
  \text{(full product, $x_n = a_n$)} \\
z^0_b &:= c_n b_n + S^+\,, &&
  \text{(zero product, $x_n = b_n$)} \\
z^0_a &:= c_n a_n + S^+\,. &&
  \text{(zero product, $x_n = a_n$)}
\end{alignat*}

\medskip

\item \textbf{Step 3: Return.}
Let the optimal objective value $z^* := \max\{z^{\Pi}_b,\; z^{\Pi}_a,\;
z^0_b,\; z^0_a\}$, and return the corresponding extreme point:

\medskip
\emph{Case A: $N_n^- = \emptyset$:}

\smallskip\noindent
If $z^* = z^{\Pi}_b$ or $z^* = z^{\Pi}_a$:
set $\bar{x}_i = b_i$ for all $i \in N_n$\,,
$\bar{x}_n = b_n$ or $a_n$ accordingly, and
$\bar{y} = \bar{x}_n \Pi$.

\smallskip\noindent
If $z^* = z^0_b$ or $z^* = z^0_a$\,:
set $\bar{x}_i = b_i$ for $i \in N_{kn}$,
$\bar{x}_k = 0$,
$\bar{x}_n = b_n$ or $a_n$ accordingly, and $\bar{y} = 0$.

\medskip
\emph{Case B: $N_n^- \neq \emptyset$:}

\smallskip\noindent
If $z^* = z^{\Pi}_b$ or $z^* = z^{\Pi}_a$\,:
set $\bar{x}_i = b_i$ for all $i \in N_n$\,,
$\bar{x}_n = b_n$ or $a_n$ accordingly, and
$\bar{y} = \bar{x}_n \Pi$.

\smallskip\noindent
If $z^* = z^0_b$ or $z^* = z^0_a$\,:
set $\bar{x}_i = b_i$ for $i \in N_n^+$
and $\bar{x}_i = 0$ for $i \in N_n^-$,
$\bar{x}_n = b_n$ or $a_n$ accordingly, and $\bar{y} = 0$.

\end{itemize}    

\medskip

In every case, the output $(\bar{x}, \bar{y})$ is an
extreme point of~$C^1_n$\,. Because the inequalities in the system of
Theorem~1 are valid for~$C^1_n$ (established above),
$(\bar{x}, \bar{y})$ is feasible for~(P), with objective
value~$z^*$.  In what follows, we construct a dual feasible solution with objective function value equal to $z^*$ for each possible outcome of the primal algorithm. 
           
\paragraph*{Dual Problem} 

Consider the dual linear program, written so that all dual variables are nonnegative:
\begin{align*}
\text{Minimize} \quad \quad &(n\!-\!2)\,a_n\Pi\, u_1 + (n\!-\!1)\,b_n\Pi\, u_2 - a_n\Pi \!\sum_{i\in N_n}\! v_i - a_n\, s_2 + \sum_{i\in N} b_i\, t_i& &   \tag{D}\label{dual}\\
\text{s.t.}\quad \quad &{-}u_1 - u_2 + \sum_{i\in N_n}v_i + \sum_{i\in N_n}w_i- s_1 = c_0\,;& \tag{$D_y$}\label{ycon}\\
\quad\quad &\frac{a_n}{b_i}\,\Pi\,u_1 + \frac{b_n}{b_i}\,\Pi\,u_2 - \frac{a_n}{b_i}\,\Pi\,v_i - \frac{b_n}{b_i}\,\Pi\,w_i + t_i = c_i\,,
&i\in N_n\,;\tag{$D_{x_i}$}\label{xicons}\\
\quad\quad &\Pi\, u_2 - \Pi\!\sum_{i\in N_n} v_i - s_2 + t_n = c_n\,;& \tag{$D_{x_n}$}\label{xncon}\\
\quad\quad &u_1,\, u_2,\, s_1,\, s_2 \geq 0\,;&\\
\quad\quad &v_i,\, w_i \geq 0\,, &i\in N_n\,;\\
\quad\quad &t_i \geq 0\,, &i\in N\,.
\end{align*}

We have seen that in each case, the primal algorithm produces four candidate objective values and the optimal solution is determined by pairwise comparisons among them. Each such comparison is equivalent to a simple inequality in the problem data.  For example, $z^\Pi_b \geq z^\Pi_a$ if and only if $c_n + c_0\Pi \geq 0$. Figures~\ref{fig:treeA} and~\ref{fig:treeB} organize these comparisons into decision trees whose leaves identify the winning candidate solution, the optimal objective value, and which particular dual certificate we use to establish optimality. In total, we require nine dual certificates: four under Case~B and five under Case~A, where the additional certificate arises because when $z^\Pi_a$ is optimal, different dual constructions are required depending on the sign of~$c_0$\,.

\paragraph*{Dual Certificates}

In what follows, when specifying dual solutions, we (mostly) provide only the variables that are not explicitly set to zero. We also make use of the notation $\gamma^+ := \max\{0,\, \gamma\}$ and $\gamma^- := \max\{0,\, -\gamma\}$ for $\gamma \in \R$, noting that $\gamma^+, \gamma^- \geq 0$ and $\gamma^+ - \gamma^- = \gamma$.

\bigskip
We begin with \textbf{Case A} where $N_n^- = \emptyset$ (i.e., all $c_i \geq 0$ for $i \in N_n$). Note that in this case $S^- = 0$, $S^+ = \sum_{i \in N_n} c_i b_i \geq 0$, and $k := \argmin_{j \in N_n^+} c_j b_j = \argmin_{j \in N_n} c_j b_j$\,.  We present five dual certificates, one for each possible primal optimal objective value with the $z^{\Pi}_a$ case split on the sign of $c_0$\,.

\begin{figure*}[p!]
    \centering
    \begin{subfigure}[t]{\textwidth}
        \centering
        \includegraphics[width=0.82\textwidth]{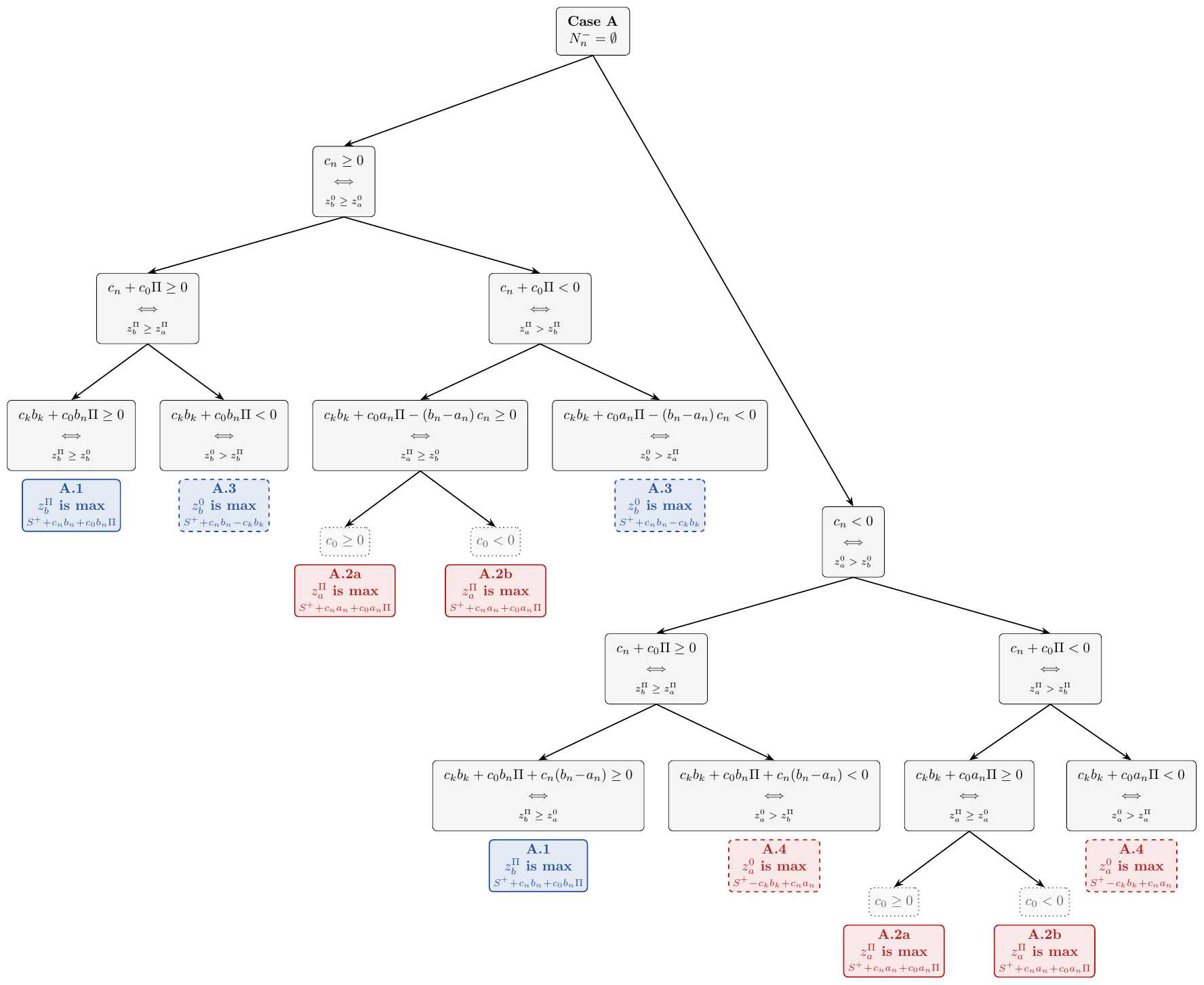}
        \caption{Tree for Case A}
        \label{fig:treeA}
    \end{subfigure}%
    
    \begin{subfigure}[t]{\textwidth}
        \centering
        \includegraphics[width=0.82\textwidth]{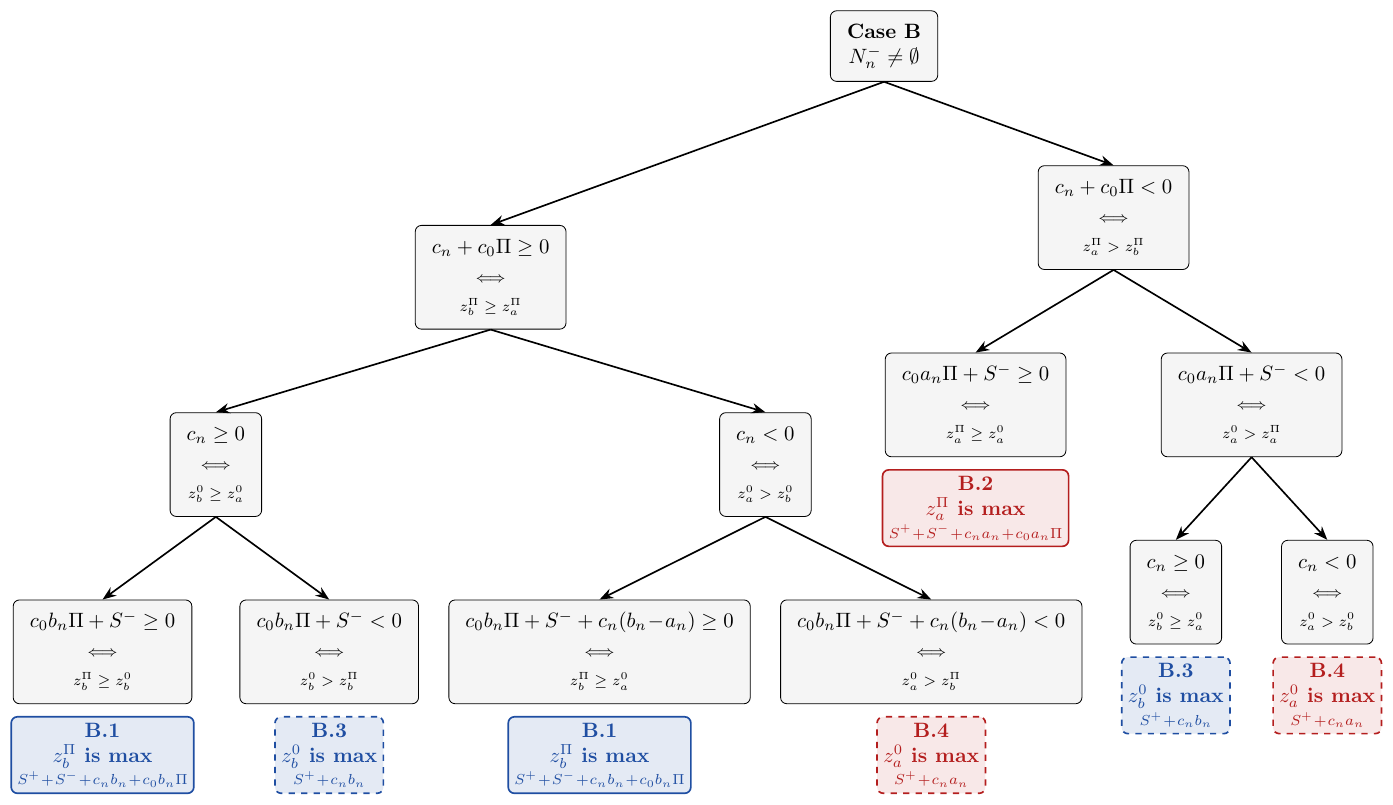}        
        \caption{Tree for Case B}
           \label{fig:treeB}
    \end{subfigure}
    
     \hfill \includegraphics[height=0.5in]{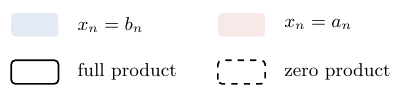}
    \caption{Decision trees for selecting the dual certificate in the proof of Theorem~\ref{thm:facets}. Each internal node tests an inequality that is equivalent to a pairwise comparison between two candidate objective values. Leaves identify the certificate used and the optimal objective value $z^*$. In Case~A, the gray dashed nodes mark the $c_0$ sign split, which determines the dual construction but does not change the primal optimum.}
     \label{fig:casetree}
\end{figure*}

\vspace{5mm}
\textbf{Certificate A.1.}
This certificate applies when $z^{\Pi}_b$ achieves the maximum; therefore we know the primal objective value of the returned solution is $S^+ + c_n b_n + c_0 b_n \Pi$. Moreover, we have
\begin{alignat}{2}
z^{\Pi}_b \geq z^{\Pi}_a &\iff c_n + c_0\Pi \geq 0\,,\label{A1condition1}\\
\text{and}\quad z^{\Pi}_b \geq z^0_b &\iff c_kb_k + c_0 b_n\Pi \geq 0\,.\label{A1condition2}
\end{alignat}
(The third comparison $z^{\Pi}_b \geq z^0_a$ provides no
additional information.)
Now consider the following solution to \eqref{dual} with some fixed $\ell \in N_n$\,:
\begin{itemize}
  \item $u_2 := c_0^-$;
  \item $v_\ell := c_0^+$;
  \item $\displaystyle t_\ell := c_\ell
        + \frac{a_n}{b_\ell}\,\Pi\,c_0^+
        - \frac{b_n}{b_\ell}\,\Pi\,c_0^-$;
  \item $t_n := c_n + c_0 \Pi$;
  \item $\displaystyle t_i := c_i
        - \frac{b_n}{b_i}\,\Pi\,c_0^-$,
        \quad for $i \in N_{\ell n}$\,.
\end{itemize}
\noindent
To demonstrate feasibility, we verify each dual constraint and nonnegativity condition.

\medskip\noindent
\emph{Dual constraints.}
\begin{itemize}
\item ($D_y$):
$$-u_2 + v_\ell = -c_0^- + c_0^+ = c_0\,.$$
\item ($D_{x_i}$) for $i \in N_{\ell n}$:
$$
  \frac{b_n}{b_i}\,\Pi\,u_2 + t_i
  \;=\; \frac{b_n}{b_i}\,\Pi\,c_0^-
        + c_i
        - \frac{b_n}{b_i}\,\Pi\,c_0^-
  \;=\; c_i\,.
$$
\item ($D_{x_\ell}$):
$$
  \frac{b_n}{b_\ell}\,\Pi\,u_2
  - \frac{a_n}{b_\ell}\,\Pi\,v_\ell
  + t_\ell
  \;=\; \frac{b_n}{b_\ell}\,\Pi\,c_0^-
        - \frac{a_n}{b_\ell}\,\Pi\,c_0^+
        + c_\ell
        + \frac{a_n}{b_\ell}\,\Pi\,c_0^+
        - \frac{b_n}{b_\ell}\,\Pi\,c_0^-
  \;=\; c_\ell\,.
$$
\item ($D_{x_n}$):
$$
  \Pi\, u_2 - \Pi\, v_\ell + t_n
  \;=\; \Pi\,c_0^- - \Pi\, c_0^+ + c_n + c_0 \Pi
  \;=\; \Pi\bigl(c_0^- - c_0^+ + c_0\bigr) + c_n
  \;=\; c_n\,.
$$
\end{itemize}

\noindent
\emph{Nonnegativity.}
\begin{itemize}
\item $u_2 = c_0^- \geq 0$;
\item $v_\ell = c_0^+ \geq 0$;
\item $t_n = c_n + c_0 \Pi \geq 0$
      by \eqref{A1condition1};
\item $t_i$ for $i \in N_{\ell n}$\,:

      When $c_0 \geq 0$: $\;t_i = c_i \geq 0$.

      When $c_0 < 0$:
      $$
        t_i = \frac{c_i b_i + c_0 b_n \Pi}{b_i}
        \;\geq\; \frac{c_k b_k + c_0 b_n \Pi}{b_i}
        \;\geq\; 0\,,
      $$
      using $c_i b_i \geq c_k b_k$ by definition of $k$,
      and $c_k b_k + c_0 b_n \Pi \geq 0$ by \eqref{A1condition2};

\item $t_\ell$\,:

      When $c_0 \geq 0$:
      $\;t_\ell = c_\ell + \frac{a_n}{b_\ell}\,\Pi\,c_0^+ \geq 0$.

      When $c_0 < 0$:
      $$
        t_\ell = \frac{c_\ell b_\ell + c_0 b_n \Pi}{b_\ell}
        \;\geq\; 0
      $$
      (same reasoning as above).
\end{itemize}

\smallskip
\noindent
\emph{Objective value.}

Finally, we verify that the dual objective value equals the primal objective $z^* = S^+ + c_n b_n + c_0 b_n \Pi$.
\begin{align*}
  &(n\!-\!1)\,b_n \Pi \cdot u_2 - a_n \Pi \cdot v_\ell
    + b_\ell\, t_\ell
    + \sum_{i \in N_{\ell n}} b_i\, t_i
    + b_n\, t_n \\[4pt]
  &\quad = (n\!-\!1)\,b_n \Pi\,c_0^-
    - a_n \Pi\,c_0^+
    + c_\ell b_\ell + a_n \Pi\,c_0^+ - b_n \Pi\,c_0^- \\
  &\qquad + \sum_{i \in N_{\ell n}}  (c_i b_i
    - b_n\,\Pi\,c_0^-)
    + c_n b_n + c_0 b_n \Pi \\[4pt]
  &\quad = (n\!-\!1)\,b_n \Pi\,c_0^-
    - a_n \Pi\,c_0^+
    + c_\ell b_\ell + a_n \Pi\,c_0^+ - b_n \Pi\,c_0^- \\
  &\qquad + (S^+ - c_\ell b_\ell)
    - (n\!-\!2)\,b_n \Pi\,c_0^-
    + c_n b_n + c_0 b_n \Pi \\[4pt]
  &\quad = S^+ + c_n b_n + c_0 b_n \Pi\,.
\end{align*}
This equals the primal objective, so the primal solution returned by the algorithm in this case: $x_i = b_i$ for all $i \in N$ and $y = b_n \Pi$, is optimal.

\vspace{5mm}
\textbf{Certificate A.2a.}
This certificate applies when $z^{\Pi}_a$ achieves the maximum (and we have $c_0 \geq 0$); therefore we know the primal objective value of the returned solution is $S^+ + c_n a_n + c_0 a_n \Pi$. Moreover, we have
\begin{alignat}{2}
z^{\Pi}_a \geq z^{\Pi}_b &\iff c_n + c_0\Pi \leq 0\,.\label{A2acondition1}
\end{alignat}
(The comparisons $z^{\Pi}_a \geq z^0_b$ and $z^{\Pi}_a \geq z^0_a$ are automatic because $c_0 \geq 0$ and all $c_i \geq 0$.)
Now consider the following solution to \eqref{dual} with some fixed $\ell \in N_n$:
\begin{itemize}
  \item $v_\ell := c_0$\,;
  \item $\displaystyle t_\ell := c_\ell + \frac{c_0 a_n}{b_\ell}\,\Pi$;
  \item $t_i := c_i$\,, \quad for $i \in N_{\ell n}$\,;
  \item $t_n := 0$;
  \item $s_2 := -(c_n + c_0\Pi)$.
\end{itemize}
\noindent
To demonstrate feasibility, we verify each dual constraint and nonnegativity condition.

\medskip\noindent
\emph{Dual constraints.}
\begin{itemize}
\item ($D_y$):
$$v_\ell = c_0.$$
\item ($D_{x_\ell}$):
$$
  -\frac{a_n}{b_\ell}\,\Pi\,v_\ell + t_\ell
  \;=\; -\frac{c_0 a_n}{b_\ell}\,\Pi
        + c_\ell
        + \frac{c_0 a_n}{b_\ell}\,\Pi
  \;=\; c_\ell\,.
$$
\item ($D_{x_i}$) for $i \in N_{\ell n}$\,:
$$t_i \;=\; c_i\,.$$
\item ($D_{x_n}$):
$$
  -\Pi\, v_\ell - s_2 + t_n
  \;=\; -c_0\Pi + (c_n + c_0\Pi) + 0
  \;=\; c_n\,.
$$
\end{itemize}

\noindent
\emph{Nonnegativity.}
\begin{itemize}
\item $v_\ell = c_0 \geq 0$ by assumption;
\item $t_i = c_i \geq 0$ for $i \in N_{\ell n}$\,;
\item $\displaystyle t_\ell = c_\ell + \frac{c_0 a_n}{b_\ell}\,\Pi \geq 0$ (both terms are nonnegative);
\item $s_2 = -(c_n + c_0\Pi) \geq 0$
      by \eqref{A2acondition1}.
\end{itemize}

\smallskip
\noindent
\emph{Objective value.}

Finally, we verify that the dual objective value equals the primal objective $z^* = S^+ + c_n a_n + c_0 a_n \Pi$.
\begin{align*}
  &-a_n \Pi \cdot v_\ell - a_n\, s_2
    + b_\ell\, t_\ell
    + \sum_{i \in N_{\ell n}} b_i\, t_i \\[4pt]
  &\quad = -c_0 a_n \Pi
    + a_n(c_n + c_0\Pi)
    + b_\ell\!\left(c_\ell + \frac{c_0 a_n}{b_\ell}\,\Pi\right)
    + (S^+ - c_\ell b_\ell) \\[4pt]
  &\quad = -c_0 a_n \Pi
    + a_n c_n + a_n c_0\Pi
    + c_\ell b_\ell + c_0 a_n\Pi
    + S^+ - c_\ell b_\ell \\[4pt]
  &\quad = S^+ + c_n a_n + c_0 a_n \Pi\,.
\end{align*}
This equals the primal objective, so the primal solution returned by the algorithm in this case: $x_i = b_i$ for all $i \in N_n$\,, $x_n = a_n$\,, and $y = a_n \Pi$, is optimal.

\vspace{5mm}
\textbf{Certificate A.2b.}
This certificate applies when $z^{\Pi}_a$ achieves the maximum (and we have $c_0 < 0$); therefore we know the primal objective value of the returned solution is $S^+ + c_n a_n + c_0 a_n \Pi$. Moreover, we have
\begin{alignat}{2}
z^{\Pi}_a \geq z^{\Pi}_b &\iff c_n + c_0\Pi \leq 0\,,\label{A2bcondition1}\\
z^{\Pi}_a \geq z^0_a &\iff c_kb_k + c_0 a_n\Pi \geq 0\,,\label{A2bcondition2}\\
\text{and}\quad z^{\Pi}_a \geq z^0_b &\iff c_kb_k + c_0 a_n\Pi - (b_n\!-\!a_n)\,c_n \geq 0\,.\label{A2bcondition3}
\end{alignat}
Of conditions \eqref{A2bcondition2} and \eqref{A2bcondition3}, \eqref{A2bcondition3} is stronger when $c_n \geq 0$ and \eqref{A2bcondition2} is stronger when $c_n < 0$; both are captured by \begin{equation}
\label{A2bcondition4}
c_kb_k + c_0 a_n\Pi - (b_n - a_n)\,c_n^+ \geq 0.
\end{equation}
Now consider the following solution to \eqref{dual}:
\begin{itemize}
  \item $\displaystyle u_1 := -c_0 - \frac{c_n^+}{\Pi}$;
  \item $\displaystyle u_2 := \frac{c_n^+}{\Pi}$;
  \item $\displaystyle t_i := c_i
        - \frac{a_n\Pi\, u_1 + b_n\Pi\, u_2}{b_i} =  \frac{c_i b_i + c_0 a_n \Pi - (b_n - a_n)\,c_n^+}{b_i}$,
        \quad for $i \in N_n$\,;
  \item $t_n := 0$;
  \item $s_2 := c_n^-$\,.
\end{itemize}

\noindent
To demonstrate feasibility, we verify each dual constraint and nonnegativity condition.

\medskip\noindent
\emph{Dual constraints.}
\begin{itemize}
\item ($D_y$):
$$-u_1 - u_2 = c_0 + \frac{c_n^+}{\Pi}-\frac{c_n^+}{\Pi}  = c_0\,.$$
\item ($D_{x_i}$) for $i \in N_n$\,:
$$
  \frac{a_n}{b_i}\,\Pi\,u_1
  + \frac{b_n}{b_i}\,\Pi\,u_2
  + t_i
  \;=\; \frac{a_n\Pi\, u_1 + b_n\Pi\, u_2}{b_i}
        + c_i
        - \frac{a_n\Pi\, u_1 + b_n\Pi\, u_2}{b_i}
  \;=\; c_i\,.
$$
\item ($D_{x_n}$):
$$
  \Pi\, u_2 - s_2 + t_n
  \;=\; c_n^+ - c_n^- + 0
  \;=\; c_n\,.
$$
\end{itemize}

\noindent
\emph{Nonnegativity.}
\begin{itemize}
\item $u_1 = -c_0 - c_n^+/\Pi \geq 0 \;\Longleftrightarrow\; c_n + c_0\Pi \leq 0$, which holds by \eqref{A2bcondition1};
\item $u_2 = c_n^+/\Pi \geq 0$;
\item $s_2 = c_n^- \geq 0$;
\item For $i \in N_n$\,:
$$
        t_i = \frac{c_i b_i + c_0 a_n \Pi - (b_n - a_n)\,c_n^+}{b_i}
        \;\geq\; \frac{c_k b_k + c_0 a_n \Pi - (b_n - a_n)\,c_n^+}{b_i}
        \;\geq\; 0\,,
$$
      because $c_i b_i \geq c_k b_k$ by definition of $k$
      and $c_kb_k + c_0 a_n\Pi - (b_n - a_n)\,c_n^+ \geq 0$ by \eqref{A2bcondition4}.
\end{itemize}

\smallskip
\noindent
\emph{Objective value.}

Finally, we verify that the dual objective value equals the primal objective $z^* = S^+ + c_n a_n + c_0 a_n \Pi$.
\begin{align*}
  &(n\!-\!2)\,a_n\Pi \cdot u_1
    + (n\!-\!1)\,b_n\Pi \cdot u_2
    - a_n\, s_2
    + \sum_{i \in N_n} b_i\, t_i \\[4pt]
  &\quad = (n\!-\!2)\,a_n\Pi\, u_1
    + (n\!-\!1)\,b_n\Pi\, u_2
    - a_n\,c_n^-
    + S^+ - (n\!-\!1)\,(a_n\Pi\, u_1 + b_n\Pi\, u_2) \\[4pt]
  &\quad = -a_n\Pi\, u_1 - a_n\,c_n^- + S^+ \\[4pt]
  &\quad = -a_n\Pi\!\left(-c_0 - \frac{c_n^+}{\Pi}\right) - a_n\,c_n^- + S^+ \\[4pt]
  &\quad = c_0 a_n\Pi + a_n\,c_n^+ - a_n\,c_n^- + S^+ \\[4pt]
  &\quad = c_0 a_n\Pi + a_n\,c_n + S^+\,.
\end{align*}
This equals the primal objective, so the primal solution returned by the algorithm in this case: $x_i = b_i$ for all $i \in N_n$\,, $x_n = a_n$\,, and $y = a_n \Pi$, is optimal.

\vspace{5mm}
\textbf{Certificate A.3.}
This certificate applies when $z^0_b$ achieves the maximum; therefore we know the primal objective value of the returned solution is $S^+ + c_n b_n - c_k b_k$\,. Moreover, we have
\begin{alignat}{2}
z^0_b \geq z^0_a &\iff c_n \geq 0\,,\label{A3condition1}\\
z^0_b \geq z^{\Pi}_b &\iff c_0 b_n\Pi + c_kb_k \leq 0\,,\label{A3condition2}\\
\text{and}\quad z^0_b \geq z^{\Pi}_a &\iff - c_kb_k -c_0 a_n\Pi  + c_n(b_n - a_n) \geq 0\,.\label{A3condition3}
\end{alignat}

Now let
$$u_1 := \frac{(c_kb_k - c_nb_n)^+}{a_n\Pi} \quad \text{and} \quad
  u_2 := \frac{\min\left \{c_n,\; c_kb_k/b_n\right \}}{\Pi}\,.$$
We have $u_1 \geq 0$ by definition, we also have $u_2 \geq 0$ because $c_n \geq 0$ (by \eqref{A3condition1}) and $c_kb_k/b_n \geq 0$.

Consider the expression $$a_n\Pi\, u_1 + b_n\Pi\, u_2,$$  
and note that if $c_kb_k \geq c_nb_n$\,, then 
$$u_1 = (c_kb_k - c_nb_n)/(a_n\Pi) \quad \text{and} \quad u_2 = c_n/\Pi,$$ giving
$$
  a_n\Pi\, u_1 + b_n\Pi\, u_2
  \;=\; (c_kb_k - c_nb_n) + c_nb_n
  \;=\; c_kb_k\,.
$$
Moreover, if $c_kb_k < c_nb_n$\,, then 
$$u_1 = 0 \quad \text{and} \quad u_2 = c_kb_k/(b_n\Pi),$$ 
giving
$$
  a_n\Pi\, u_1 + b_n\Pi\, u_2
  \;=\; 0 + c_kb_k
  \;=\; c_kb_k\,.
$$

Importantly, in both cases, we have 
\begin{equation}
\label{A3keyidentity}
    a_n\Pi\, u_1 + b_n\Pi\, u_2 = c_kb_k\,.
\end{equation}

Now consider the following solution to \eqref{dual}:
\begin{itemize}
  \item $u_1$, $u_2$ as defined above;
  \item $\displaystyle t_i := c_i - \frac{c_kb_k}{b_i}$,
        \quad for $i \in N_{kn}$\,;
  \item $t_k := 0$;
  \item $\displaystyle t_n := \frac{(c_nb_n - c_kb_k)^+}{b_n}$;
  \item $s_1 := -c_0 - u_1 - u_2$\,.
\end{itemize}
\noindent
To demonstrate feasibility, we verify each dual constraint and nonnegativity condition.

\medskip
\noindent
\emph{Dual constraints.}
\begin{itemize}
\item ($D_y$):
$$-u_1 - u_2 - s_1 = -u_1 - u_2 - (-c_0 - u_1 - u_2) = c_0\,.$$
\item ($D_{x_k}$): using \eqref{A3keyidentity},
$$
  \frac{a_n}{b_k}\,\Pi\,u_1
  + \frac{b_n}{b_k}\,\Pi\,u_2
  + t_k
  \;=\; \frac{a_n\Pi\, u_1 + b_n\Pi\, u_2}{b_k} + 0
  \;=\; \frac{c_kb_k}{b_k}
  \;=\; c_k\,.
$$
\item ($D_{x_i}$) for $i \in N_{kn}$: again using \eqref{A3keyidentity},
$$
  \frac{a_n}{b_i}\,\Pi\,u_1
  + \frac{b_n}{b_i}\,\Pi\,u_2
  + t_i
  \;=\; \frac{c_kb_k}{b_i}
        + c_i
        - \frac{c_kb_k}{b_i}
  \;=\; c_i\,.
$$
\item ($D_{x_n}$): 

If $c_nb_n \geq c_kb_k$\,:
$$
  \Pi\, u_2 + t_n
  \;=\; \frac{c_kb_k}{b_n} + \frac{c_nb_n - c_kb_k}{b_n}
  \;=\; c_n\,.
$$
If $c_nb_n < c_kb_k$\,:
$$
  \Pi\, u_2 + t_n
  \;=\; c_n + 0
  \;=\; c_n\,.
$$
\end{itemize}

\noindent
\emph{Nonnegativity.}
\begin{itemize}
\item $u_1 \geq 0$ and $u_2 \geq 0$ (established above);
\item $\displaystyle t_n = \frac{(c_nb_n - c_kb_k)^+}{b_n} \geq 0$;
\item For $i \in N_{kn}$\,:
      $\displaystyle t_i = c_i - \frac{c_kb_k}{b_i} = \frac{c_ib_i - c_kb_k}{b_i} \geq 0$,
      by the definition of $k$;
\item $s_1 = -c_0 - u_1 - u_2 \geq 0$:

      If $c_kb_k \geq c_nb_n$\,: 
$$s_1 = -c_0 - \frac{c_kb_k-c_nb_n}{a_n\Pi} - \frac{c_n}{\Pi} = \frac{-c_0a_n\Pi - c_kb_k+ c_n(b_n-a_n)}{a_n\Pi} \geq 0,$$
     by \eqref{A3condition3}. 
      If $c_kb_k < c_nb_n$\,:     
$$s_1 = -c_0 - 0 - \frac{c_kb_k}{b_n\Pi} = \frac{-c_0b_n\Pi - c_kb_k}{b_n\Pi} \geq 0,$$ 
     by \eqref{A3condition2}. 
\end{itemize}

\smallskip
\noindent
\emph{Objective value.}

Finally, we verify that the dual objective value equals the primal objective $z^* = S^+ + c_n b_n - c_k b_k$\,.

\begin{align*}
  &(n\!-\!2)\,a_n\Pi\, u_1
    + (n\!-\!1)\,b_n\Pi\, u_2
    + \sum_{i \in N_{kn}} b_i\, t_i
    + b_n\, t_n \\[4pt]
  &\quad \quad= (n\!-\!2)(c_kb_k - c_nb_n)^+
    + (n\!-\!1)\,b_n\min\left\{c_n, \, \frac{c_kb_k}{b_n}\right\}\\
  &\quad \quad\quad \quad \quad \quad + \sum_{i \in N_{kn}} b_i(c_i - \frac{c_kb_k}{b_i}) + (c_nb_n - c_kb_k)^+.
\end{align*}

When $c_kb_k \geq c_nb_n$\,, this is equal to:
\begin{align*}
 &(n\!-\!2)(c_kb_k - c_nb_n)
    + (n\!-\!1)\,c_nb_n
    + (S^+ - c_kb_k) - (n\!-\!2)\,c_kb_k
    + 0 \\[4pt]
  &\quad = (n\!-\!2)\,c_kb_k - (n\!-\!2)\,c_nb_n
    + (n\!-\!1)\,c_nb_n
    + S^+ - c_kb_k - (n\!-\!2)\,c_kb_k \\[4pt]
  &\quad = S^+ + c_nb_n - c_kb_k\,.
\end{align*}

When $c_kb_k < c_nb_n$\,, this is equal to:
\begin{align*}
  &0
    + (n\!-\!1)\,c_kb_k
    + (S^+ - c_kb_k) - (n\!-\!2)\,c_kb_k
    + (c_nb_n - c_kb_k) \\[4pt]
  &\quad = S^+ + c_nb_n - c_kb_k\,.
\end{align*}
In both cases, this equals the primal objective, so the primal solution returned by the algorithm in this case: $x_i = b_i$ for $i \in N_{kn}$\,, $x_k = 0$, $x_n = b_n$\,, and $y = 0$, is optimal.

\vspace{5mm}
\textbf{Certificate A.4.}
This certificate applies when $z^0_a$ achieves the maximum; therefore we know the primal objective value of the returned solution is $S^+ - c_kb_k + c_na_n$\,. Moreover, we have
\begin{alignat}{2}
z^0_a \geq z^0_b &\iff c_n \leq 0\,,\label{A4condition1}\\
\text{and}\quad z^0_a \geq z^{\Pi}_a &\iff c_kb_k + c_0 a_n\Pi \leq 0\,.\label{A4condition2}
\end{alignat}
(The third comparison $z^0_a \geq z^{\Pi}_b$ provides no additional information.)
Now consider the following solution to \eqref{dual}:
\begin{itemize}
  \item $\displaystyle u_1 := \frac{c_kb_k}{a_n\Pi}$;
  \item $\displaystyle t_i := c_i - \frac{c_kb_k}{b_i}$,
        \quad for $i \in N_{kn}$\,;
  \item $t_k := 0$;
  \item $t_n := 0$;
  \item $\displaystyle s_1 := \frac{-(c_kb_k + c_0 a_n\Pi)}{a_n\Pi}$;
  \item $s_2 := -c_n$\,.
\end{itemize}
\noindent
To demonstrate feasibility, we verify each dual constraint and nonnegativity condition.

\medskip
\noindent
\emph{Dual constraints.}
\begin{itemize}
\item ($D_y$):
$$-u_1 - s_1 = -\frac{c_kb_k}{a_n\Pi} + \frac{c_kb_k + c_0 a_n\Pi}{a_n\Pi} = c_0\,.$$
\item ($D_{x_k}$):
$$
  \frac{a_n}{b_k}\,\Pi\,u_1 + t_k
  \;=\; \frac{a_n}{b_k}\,\Pi\,\frac{c_kb_k}{a_n\Pi} + 0
  \;=\; c_k\,.
$$
\item ($D_{x_i}$) for $i \in N_{kn}$\,:
$$
  \frac{a_n}{b_i}\,\Pi\,u_1 + t_i
  \;=\; \frac{c_kb_k}{b_i}
        + c_i
        - \frac{c_kb_k}{b_i}
  \;=\; c_i\,.
$$
\item ($D_{x_n}$):
$$
  -s_2 + t_n
  \;=\; c_n + 0
  \;=\; c_n\,.
$$
\end{itemize}

\noindent
\emph{Nonnegativity.}
\begin{itemize}
\item $u_1 = c_kb_k/(a_n\Pi) \geq 0$,
      because $c_k \geq 0$;
\item For $i \in N_{kn}$\,:
      $t_i = (c_ib_i - c_kb_k)/b_i \geq 0$
      by definition of $k$;
\item $s_2 = -c_n \geq 0$
      by \eqref{A4condition1};
\item $s_1 = -(c_kb_k + c_0 a_n\Pi)/(a_n\Pi) \geq 0$
      by \eqref{A4condition2}.
\end{itemize}

\smallskip
\noindent
\emph{Objective value.}

Finally, we verify that the dual objective value equals the primal objective $z^* = S^+ - c_kb_k + c_na_n$\,.
\begin{align*}
  &(n\!-\!2)\,a_n\Pi \cdot u_1 - a_n\, s_2
    + \sum_{i \in N_{kn}} b_i\, t_i \\[4pt]
  &\quad = (n\!-\!2)\,c_kb_k
    + a_nc_n
    + \sum_{i \in N_{kn}} (c_ib_i - c_kb_k) \\[4pt]
  &\quad = (n\!-\!2)\,c_kb_k
    + a_nc_n
    + (S^+ - c_kb_k) - (n\!-\!2)\,c_kb_k \\[4pt]
  &\quad = S^+ - c_kb_k + c_na_n\,.
\end{align*}
This equals the primal objective, so the primal solution returned by the algorithm in this case: $x_i = b_i$ for $i \in N_{kn}$\,, $x_k = 0$, $x_n = a_n$\,, and $y = 0$, is optimal.


\bigskip

We now move to \textbf{Case B} where $N_n^- \neq \emptyset$ (i.e., at least one $c_i < 0$ for $i \in N_n$).  We first note that in this  case $S^- <0$ and all certificates have $u_1 = u_2 = 0$, which simplifies the dual constraints to:
\begin{align*}
(D_y)\!:\; &\sum_{i\in N_n}(v_i + w_i) - s_1 = c_0\,;\\
(D_{x_i})\!:\; &-\frac{a_n}{b_i}\,\Pi\, v_i - \frac{b_n}{b_i}\,\Pi\, w_i + t_i = c_i\,;\\
(D_{x_n})\!:\; &-\Pi\!\sum_{i\in N_n} v_i - s_2 + t_n = c_n\,.
\end{align*}
We present four dual certificates, one for each possible primal optimal objective value.

\vspace{5mm}
\textbf{Certificate B.1.}
This certificate applies when $z^{\Pi}_b$ achieves the maximum; therefore we know the primal objective value of the returned solution is $S^+ + S^- + c_n b_n + c_0 b_n \Pi$. Moreover, we have
\begin{alignat}{2}
z^{\Pi}_b \geq z^{\Pi}_a &\iff c_n + c_0\Pi \geq 0\,,\label{B1condition1}\\
z^{\Pi}_b \geq z^0_b &\iff c_0 b_n\Pi + S^- \geq 0\,,\label{B1condition2}\\
\text{and}\quad z^{\Pi}_b \geq z^0_a &\iff c_0 b_n\Pi + S^- + c_n(b_n\!-\!a_n) \geq 0\,.\label{B1condition3}
\end{alignat}
(When $c_n \geq 0$, \eqref{B1condition3} follows from \eqref{B1condition2};
when $c_n < 0$, \eqref{B1condition2} follows from \eqref{B1condition3}.)

Now let
$$\sigma := \frac{c_n^-}{\Pi} \quad \text{and} \quad
  \lambda_i := \frac{c_i b_i}{S^-} \;\;\text{for } i \in N_n^-\,.$$
When $c_n \geq 0$: $\sigma = 0$.
When $c_n < 0$: $\sigma = -c_n/\Pi > 0$.
Moreover, $\lambda_i \geq 0$ for each $i \in N_n^-$
because $c_i b_i < 0$ and $S^- < 0$. We also have
$$
  \sum_{i \in N_n^-} \lambda_i
  \;=\; \frac{1}{S^-} \sum_{i \in N_n^-} c_i b_i
  \;=\; 1\,.
$$

Finally, we observe that $$c_0 - \sigma \geq 0.$$
This is because when $c_n \geq 0$, $\sigma = 0$ and $c_0 > 0$
(from \eqref{B1condition2} because $S^- < 0$);
when $c_n < 0$, $c_0 \geq \sigma = -c_n/\Pi \;\Longleftrightarrow \; c_0\Pi + c_n \geq 0$, which holds by \eqref{B1condition1}.

Now consider the following solution to \eqref{dual}:
\begin{itemize}
  \item $v_i := \sigma\,\lambda_i$\,, \quad for $i \in N_n^-$;
  \item $w_i := (c_0 - \sigma)\,\lambda_i$\,, \quad for $i \in N_n^-$;
  \item $t_i := c_i$\,, \quad for $i \in N_n^+$;
  \item $\displaystyle t_i := c_i
        + \frac{a_n}{b_i}\,\Pi\,v_i
        + \frac{b_n}{b_i}\,\Pi\,w_i \,=\, c_i\!\left(1 + \frac{\Pi\bigl(b_n c_0 - (b_n\!-\!a_n)\sigma\bigr)}{S^-}\right)$,  \quad for $i \in N_n^-$;
  \item $t_n := c_n + \Pi\sigma   $.
\end{itemize}
\noindent
To demonstrate feasibility, we verify each dual constraint and nonnegativity condition.

\medskip
\noindent
\emph{Dual constraints.}
\begin{itemize}
\item ($D_y$):
$$
  \sum_{i \in N_n^-}(v_i + w_i)
  \;=\;  \sum_{i \in N_n^-} (\sigma\lambda_i + c_0 \lambda_i - \sigma\lambda_i )
  \;=\; c_0 \sum_{i \in N_n^-} \lambda_i
   \;=\; c_0 \cdot 1
  \;=\; c_0\,.
$$
\item ($D_{x_i}$) for $i \in N_n^+$:
$$t_i \;=\; c_i\,.$$
\item ($D_{x_i}$) for $i \in N_n^-$:
$$
  -\frac{a_n}{b_i}\,\Pi\, v_i
  - \frac{b_n}{b_i}\,\Pi\, w_i
  + t_i
  \;=\; -\frac{a_n}{b_i}\,\Pi\, v_i
        - \frac{b_n}{b_i}\,\Pi\, w_i
        + c_i
        + \frac{a_n}{b_i}\,\Pi\, v_i
        + \frac{b_n}{b_i}\,\Pi\, w_i
  \;=\; c_i\,.
$$
\item ($D_{x_n}$):
$$
  -\Pi \sum_{i \in N_n^-} v_i + t_n
  \;=\; -\Pi\sigma \sum_{i \in N_n^-} \lambda_i + c_n + \Pi\sigma
  \;=\; -\Pi\sigma \cdot 1 + c_n + \Pi\sigma
  \;=\; c_n\,.
$$
\end{itemize}

\noindent
\emph{Nonnegativity.}
\begin{itemize}
\item $v_i = \sigma\,\lambda_i \geq 0$
      because $\sigma \geq 0$ and $\lambda_i \geq 0$;
\item $w_i = (c_0 - \sigma)\,\lambda_i \geq 0$
      because $c_0 - \sigma \geq 0$ (established above) and $\lambda_i \geq 0$;
\item $t_i = c_i \geq 0$ for $i \in N_n^+$;
\item $t_n = c_n + \Pi\sigma = c_n + c_n^- \geq 0$:

      When $c_n \geq 0$: $t_n = c_n \geq 0$.

      When $c_n < 0$: $t_n = c_n + (-c_n) = 0$;
\item For $i \in N_n^-$: substituting $v_i = \sigma\lambda_i$\,,
      $w_i = (c_0 - \sigma)\lambda_i$\,, and $\lambda_i = c_ib_i/S^-$, we have
      $$
        t_i
        \;=\; c_i\!\left(1 + \frac{\Pi\bigl(b_n c_0 - (b_n\!-\!a_n)\sigma\bigr)}{S^-}\right).
      $$
      Now because $c_i < 0$,
      \begin{align*}
       t_i \geq 0 &\,\Longleftrightarrow \, 1 + \frac{\Pi\bigl(b_n c_0 - (b_n\!-\!a_n)\sigma\bigr)}{S^-} \leq 0\, \\
       &\,\Longleftrightarrow \, S^- + \Pi\bigl(b_n c_0 - (b_n\!-\!a_n)\sigma\bigr) \geq 0 \\
        &\,\Longleftrightarrow \,  c_0 b_n\Pi + S^- - (b_n\!-\!a_n)\Pi\sigma \geq 0\, \\
         &\,\Longleftrightarrow \,  c_0 b_n\Pi + S^- - (b_n\!-\!a_n)c_n^-\geq 0\,.
      \end{align*}
      
      When $c_n \geq 0$: $c_n^- = 0$, so the condition reduces to
      $c_0 b_n\Pi + S^- \geq 0$, which holds by \eqref{B1condition2}.

      When $c_n < 0$: $c_n^- = -c_n$\,, so the condition becomes
      $c_0 b_n\Pi + S^- + c_n(b_n\!-\!a_n) \geq 0$, which holds by \eqref{B1condition3}.
\end{itemize}

\smallskip
\noindent
\emph{Objective value.}

Finally, we verify that the dual objective value equals the primal objective $z^* = S^+ + S^- + c_n b_n + c_0 b_n \Pi$.
\begin{align*}
  &-a_n \Pi \sum_{i \in N_n^-} v_i
    + \sum_{i \in N_n^+} b_i\, t_i
    + \sum_{i \in N_n^-} b_i\, t_i
    + b_n\, t_n \\[4pt]
  &\quad = -a_n\Pi \sum_{i \in N_n^-} v_i
    + S^+
    + \sum_{i \in N_n^-} \bigl(c_i b_i + a_n\Pi\, v_i + b_n\Pi\, w_i\bigr)
    + b_n(c_n + \Pi\sigma) \\[4pt]
  &\quad = -a_n\Pi \sum_{i \in N_n^-} v_i
    + S^+ + S^-
    + a_n\Pi \sum_{i \in N_n^-} v_i
    + b_n\Pi \sum_{i \in N_n^-} w_i
    + c_n b_n + b_n\Pi\sigma \\[4pt]
  &\quad = S^+ + S^-
    + b_n\Pi \sum_{i \in N_n^-} (c_0 - \sigma)\lambda_i
    + c_n b_n + b_n\Pi\sigma \\[4pt]
  &\quad = S^+ + S^-
    + b_n\Pi(c_0 - \sigma)
    + c_n b_n + b_n\Pi\sigma \\[4pt]
  &\quad = S^+ + S^- + c_n b_n + c_0 b_n \Pi.
\end{align*}

This equals the primal objective, so the primal solution returned by the algorithm in this case: $x_i = b_i$ for all $i \in N$ and $y = b_n \Pi$, is optimal.

\vspace{5mm}
\textbf{Certificate B.2.}
This certificate applies when $z^{\Pi}_a$ achieves the maximum; therefore we know the primal objective value of the returned solution is $S^+ + S^- + c_n a_n + c_0 a_n \Pi$. Moreover, we have
\begin{alignat}{2}
z^{\Pi}_a \geq z^{\Pi}_b &\iff c_n + c_0\Pi \leq 0\,,\label{B2condition1}\\
\text{and}\quad z^{\Pi}_a \geq z^0_a &\iff c_0 a_n\Pi + S^- \geq 0\,.\label{B2condition2}
\end{alignat}
(The remaining comparison $z^{\Pi}_a \geq z^0_b$ is not required in the feasibility argument below.)
Fix any $\ell \in N_n^-$. Now consider the following solution to \eqref{dual}:
\begin{itemize}
  \item $\displaystyle v_i := \frac{-c_i b_i}{a_n \Pi}$,
        \quad for $i \in N_n^- - \ell$;
  \item $\displaystyle v_{\ell} := \frac{c_0 a_n \Pi + S^- - c_{\ell} b_{\ell}}{a_n \Pi}$;
  \item $t_i := c_i$\,, \quad for $i \in N_n^+$;
  \item $t_i := 0$, \quad for $i \in N_n^- - \ell$;
  \item $\displaystyle t_{\ell} := \frac{c_0 a_n \Pi + S^-}{b_{\ell}}$;
  \item $t_n := 0$;
  \item $s_2 := -(c_n + c_0\Pi)$.
\end{itemize}
\noindent
To demonstrate feasibility, we verify each dual constraint and nonnegativity condition.

\medskip
\noindent
\emph{Dual constraints.}
\begin{itemize}
\item ($D_y$):
$$
  \sum_{i \in N_n^-} v_i
  \;=\; \sum_{i \in N_n^- - \ell} \frac{-c_i b_i}{a_n \Pi}
        + \frac{c_0 a_n \Pi + S^- - c_{\ell} b_{\ell}}{a_n \Pi}
  \;=\; \frac{-S^- + c_{\ell} b_{\ell} + c_0 a_n \Pi + S^- - c_{\ell} b_{\ell}}{a_n \Pi}
  \;=\; c_0\,.
$$
\item ($D_{x_i}$) for $i \in N_n^+$:
$$t_i \;=\; c_i\,.$$
\item ($D_{x_i}$) for $i \in N_n^- - \ell$:
$$
  -\frac{a_n}{b_i}\,\Pi\, v_i + t_i
  \;=\; -\frac{a_n}{b_i}\,\Pi\,\frac{-c_i b_i}{a_n \Pi} + 0
  \;=\; c_i\,.
$$
\item ($D_{x_{\ell}}$):
$$
  -\frac{a_n}{b_{\ell}}\,\Pi\, v_{\ell} + t_{\ell}
  \;=\; -\frac{c_0 a_n \Pi + S^- - c_{\ell} b_{\ell}}{b_{\ell}}
        + \frac{c_0 a_n \Pi + S^-}{b_{\ell}}
  \;=\; c_{\ell}\,.
$$
\item ($D_{x_n}$):
$$
  -\Pi \sum_{i \in N_n^-} v_i - s_2 + t_n
  \;=\; -c_0\Pi + (c_n + c_0\Pi) + 0
  \;=\; c_n\,.
$$
\end{itemize}

\noindent
\emph{Nonnegativity.}
\begin{itemize}
\item $\displaystyle v_i = \frac{-c_i b_i}{a_n \Pi} \geq 0$
      for $i \in N_n^- - \ell$, because $c_i < 0$;
\item $\displaystyle v_{\ell} = \frac{c_0 a_n \Pi + S^- - c_{\ell} b_{\ell}}{a_n \Pi} \geq 0$,
      because $c_0 a_n \Pi + S^- \geq 0$ by \eqref{B2condition2}
      and $c_{\ell} b_{\ell} < 0$;
\item $t_i = c_i \geq 0$ for $i \in N_n^+$;
\item $\displaystyle t_{\ell} = \frac{c_0 a_n \Pi + S^-}{b_{\ell}} \geq 0$,
      because $c_0 a_n \Pi + S^- \geq 0$ by \eqref{B2condition2};
\item $s_2 = -(c_n + c_0\Pi) \geq 0$
      by \eqref{B2condition1}.
\end{itemize}
\smallskip
\noindent
\emph{Objective value.}

Finally, we verify that the dual objective value equals the primal objective $z^* = S^+ + S^- + c_n a_n + c_0 a_n \Pi$.
\begin{align*}
  &-a_n \Pi \sum_{i \in N_n^-} v_i
    - a_n\, s_2
    + b_{\ell}\, t_{\ell}
    + \sum_{i \in N_n^+} b_i\, t_i \\[4pt]
  &\quad = -a_n \Pi \cdot c_0
    - a_n\bigl(-(c_n + c_0\Pi)\bigr)
    + (c_0 a_n \Pi + S^-)
    + S^+ \\[4pt]
  &\quad = -c_0 a_n \Pi
    + a_n c_n + a_n c_0\Pi
    + c_0 a_n \Pi + S^-
    + S^+ \\[4pt]
  &\quad = S^+ + S^- + c_n a_n + c_0 a_n \Pi.
\end{align*}
This equals the primal objective, so the primal solution returned by the algorithm in this case: $x_i = b_i$ for all $i \in N_n$\,, $x_n = a_n$\,, and $y = a_n \Pi$, is optimal.

\vspace{5mm}
\textbf{Certificate B.3.}
This certificate applies when $z^0_b$ achieves the maximum; therefore we know the primal objective value of the returned solution is $S^+ + c_n b_n$\,. Moreover, we have
\begin{alignat}{2}
z^0_b \geq z^0_a &\iff c_n \geq 0\,,\label{B3condition1}\\
\text{and}\quad z^0_b \geq z^{\Pi}_b &\iff c_0 b_n\Pi + S^- \leq 0\,.\label{B3condition2}
\end{alignat}
(The remaining comparison $z^0_b \geq z^{\Pi}_a$ holds by assumption but is not needed in the feasibility argument below.)
Now consider the following solution to \eqref{dual}:
\begin{itemize}
  \item $\displaystyle w_i := \frac{-c_i b_i}{b_n \Pi}$,
        \quad for $i \in N_n^-$;
  \item $t_i := c_i$\,, \quad for $i \in N_n^+$;
  \item $t_i := 0$, \quad for $i \in N_n^-$;
  \item $t_n := c_n$\,;
  \item $\displaystyle s_1 := \frac{-(c_0 b_n\Pi + S^-)}{b_n\Pi}$.
\end{itemize}
\noindent
To demonstrate feasibility, we verify each dual constraint and nonnegativity condition.

\medskip
\noindent
\emph{Dual constraints.}
\begin{itemize}
\item ($D_y$):
$$
  \sum_{i \in N_n^-} w_i - s_1
  \;=\; \frac{-S^-}{b_n\Pi}
        + \frac{c_0 b_n\Pi + S^-}{b_n\Pi}
  \;=\; c_0\,.
$$
\item ($D_{x_i}$) for $i \in N_n^+$:
$$t_i \;=\; c_i\,.$$
\item ($D_{x_i}$) for $i \in N_n^-$:
$$
  -\frac{b_n}{b_i}\,\Pi\, w_i + t_i
  \;=\; -\frac{b_n}{b_i}\,\Pi\,\frac{-c_i b_i}{b_n\Pi} + 0
  \;=\; c_i\,.
$$
\item ($D_{x_n}$):
$$t_n \;=\; c_n\,.$$
\end{itemize}

\noindent
\emph{Nonnegativity.}
\begin{itemize}
\item $\displaystyle w_i = \frac{-c_i b_i}{b_n \Pi} \geq 0$
      for $i \in N_n^-$, because $c_i < 0$;
\item $t_i = c_i \geq 0$ for $i \in N_n^+$;
\item $t_n = c_n \geq 0$
      by \eqref{B3condition1};
\item $\displaystyle s_1 = \frac{-(c_0 b_n\Pi + S^-)}{b_n\Pi} \geq 0$
      by \eqref{B3condition2}.
\end{itemize}

\smallskip
\noindent
\emph{Objective value.}

Finally, we verify that the dual objective value equals the primal objective $z^* = S^+ + c_n b_n$\,.

$$  \sum_{i \in N_n^+} b_i\, t_i + b_n\, t_n  = S^+ + c_n b_n\,.$$

\noindent
This equals the primal objective, so the primal solution returned by the algorithm in this case: $x_i = b_i$ for $i \in N_n^+$, $x_i = 0$ for $i \in N_n^-$, $x_n = b_n$\,, and $y = 0$, is optimal.

\vspace{5mm}
\textbf{Certificate B.4.}
This certificate applies when $z^0_a$ achieves the maximum; therefore we know the primal objective value of the returned solution is $S^+ + c_n a_n$\,. Moreover, we have
\begin{alignat}{2}
z^0_a \geq z^0_b &\iff c_n \leq 0\,,\label{B4condition1}\\
z^0_a \geq z^{\Pi}_a &\iff c_0 a_n\Pi + S^- \leq 0\,,\label{B4condition2}\\
\text{and}\quad z^0_a \geq z^{\Pi}_b &\iff c_0 b_n\Pi + S^- + c_n(b_n\!-\!a_n) \leq 0\,.\label{B4condition3}
\end{alignat}

Now let $$\beta := \min\!\left(\frac{1}{a_n\Pi},\; \frac{-c_n}{(-S^-)\Pi}\right) \quad \text{and} \quad \alpha := \frac{1 - a_n\Pi\,\beta}{b_n\Pi},$$
and note $\beta \geq 0$ because $a_n, \Pi > 0$, $-c_n \geq 0$ (by \eqref{B4condition1}), and $-S^- > 0$.
Moreover, $\alpha \geq 0$ because $\beta \leq 1/(a_n\Pi)$ by definition,
so $a_n\Pi\,\beta \leq 1$.

Also note that by construction we have
$$
  a_n\Pi\,\beta + b_n\Pi\,\alpha
  \;=\; a_n\Pi\,\beta + b_n\Pi \cdot \frac{1 - a_n\Pi\,\beta}{b_n\Pi}
  \;=\; a_n\Pi\,\beta + 1 - a_n\Pi\,\beta
  \;=\; 1\,.
$$
Finally, if $\beta = 1/(a_n\Pi)$, then
$$
  \alpha \;=\; \frac{1 - a_n\Pi \cdot \frac{1}{a_n\Pi}}{b_n\Pi}
  \;=\; \frac{1 - 1}{b_n\Pi}
  \;=\; 0\,.
$$
Whereas if $\beta = -c_n/((-S^-)\Pi)$, then
$$
  \alpha \;=\; \frac{1 - a_n\Pi \cdot \frac{-c_n}{(-S^-)\Pi}}{b_n\Pi}
  \;=\; \frac{1 + \frac{a_n c_n}{-S^-}}{b_n\Pi}
  \;=\; \frac{-S^- + a_n c_n}{(-S^-)b_n\Pi}\,,
$$
and
$$
  \alpha + \beta
  \;=\; \frac{-S^- + a_n c_n}{(-S^-)b_n\Pi} + \frac{-c_n}{(-S^-)\Pi}
  \;=\; \frac{-S^- + a_n c_n + (-c_n)b_n}{(-S^-)b_n\Pi}
  \;=\; \frac{-S^- + (a_n - b_n)c_n}{(-S^-)b_n\Pi}
  \;=\; \frac{-S^- - c_n(b_n\!-\!a_n)}{(-S^-)b_n\Pi}\,.
$$

With that in mind, we consider the following solution to \eqref{dual}:
\begin{itemize}
  \item $v_i := -\beta\, c_i b_i$\,, \quad for $i \in N_n^-$;
  \item $w_i := -\alpha\, c_i b_i$\,, \quad for $i \in N_n^-$;
  \item $t_i := c_i$\,, \quad for $i \in N_n^+$;
  \item $t_i := 0$, \quad for $i \in N_n^-$;
  \item $t_n := 0$;
  \item $s_2 := -c_n + \Pi\,\beta\, S^-$;
  \item $s_1 := -(\alpha + \beta)\,S^- - c_0$\,.
\end{itemize}
\noindent
To demonstrate feasibility, we verify each dual constraint and nonnegativity condition.

\medskip
\noindent
\emph{Dual constraints.}
\begin{itemize}
\item ($D_y$):
$$
  \sum_{i \in N_n^-}(v_i + w_i) - s_1
  \;=\; -(\beta + \alpha)\,S^- - \bigl(-(\alpha + \beta)\,S^- - c_0\bigr)
  \;=\; c_0\,.
$$
\item ($D_{x_i}$) for $i \in N_n^+$:
$$t_i \;=\; c_i\,.$$
\item ($D_{x_i}$) for $i \in N_n^-$:
$$
  -\frac{a_n}{b_i}\,\Pi\, v_i
  - \frac{b_n}{b_i}\,\Pi\, w_i
  + t_i
  \;=\; \frac{c_i b_i}{b_i}\bigl(a_n\Pi\,\beta + b_n\Pi\,\alpha\bigr) + 0
  \;=\; c_i \cdot 1
  \;=\; c_i\,.
$$
\item ($D_{x_n}$):
$$
  -\Pi \sum_{i \in N_n^-} v_i - s_2 + t_n
  \;=\; \Pi\,\beta\, S^- - (-c_n + \Pi\,\beta\, S^-) + 0
  \;=\; c_n\,.
$$
\end{itemize}

\noindent
\emph{Nonnegativity.}
\begin{itemize}
\item $v_i = -\beta\, c_i b_i \geq 0$
      because $\beta \geq 0$ and $c_i b_i < 0$ for $i \in N_n^-$;
\item $w_i = -\alpha\, c_i b_i \geq 0$
      because $\alpha \geq 0$ and $c_i b_i < 0$ for $i \in N_n^-$;
\item $t_i = c_i \geq 0$ for $i \in N_n^+$;
\item $s_2 = -c_n + \Pi\,\beta\, S^- \geq 0$:
      because $\beta \leq -c_n/((-S^-)\Pi)$ by definition, we have
      $-\Pi\,\beta\, S^- \leq -c_n\, \Longleftrightarrow \,-c_n + \Pi\,\beta\, S^- \geq 0$;
\item $s_1 = -(\alpha + \beta)\,S^- - c_0 \geq 0$:

      If $\beta = 1/(a_n\Pi)$: then $\alpha = 0$ and
      $$
        s_1 = \frac{-S^-}{a_n\Pi} - c_0
        = \frac{-(c_0 a_n\Pi + S^-)}{a_n\Pi}
       \geq 0
      $$
      by \eqref{B4condition2}.

      If $\beta = -c_n/((-S^-)\Pi)$: then
      $$
        s_1 = -\left(\frac{-S^- - c_n(b_n\!-\!a_n)}{(-S^-)b_n\Pi}\right)S^- - c_0 =\frac{-(c_0 b_n\Pi + S^- + c_n(b_n\!-\!a_n))}{b_n\Pi}
        \geq 0
      $$
      by \eqref{B4condition3}.
\end{itemize}

\smallskip
\noindent
\emph{Objective value.}

Finally, we verify that the dual objective value equals the primal objective $z^* = S^+ + c_n a_n$\,.
\begin{align*}
  &-a_n\Pi \sum_{i \in N_n^-} v_i
    - a_n\, s_2
    + \sum_{i \in N_n^+} b_i\, t_i \\[4pt]
  &\quad = a_n\Pi\,\beta\, S^-
    - a_n(-c_n + \Pi\,\beta\, S^-)
    + S^+ \\[4pt]
  &\quad = S^+ + c_n a_n\,.
\end{align*}
This equals the primal objective, so the primal solution returned by the algorithm in this case: $x_i = b_i$ for $i \in N_n^+$, $x_i = 0$ for $i \in N_n^-$, $x_n = a_n$\,, and $y = 0$, is optimal.

\medskip

In every case, the primal algorithm returns an extreme point of $C^1_n$ that is feasible for~\eqref{primal}, and the corresponding dual certificate is feasible for~\eqref{dual} with the same objective value. By strong duality, both solutions are optimal. Since the objective $c_0 y + \sum_{i \in N} c_i x_i$ was arbitrary, the inequality system in the statement of the theorem is a complete description of~$C^1_n$\,.

\end{proof}
 

\section{Volume}\label{sec:vol}

In this section, we establish a formula for the volume of  
$C_n^1$\,, for all $n$. 

\begin{theorem}\label{thm:vol}
For all $n\geq 2$,
\[
\vol{n\!+\!1}{C_n^1}
= \frac{(b_n-a_n) \prod_{i\in N_n}b_i^2}{(n+1)!} \;\bigg ( (n! - 1)b_n + \big ((n-1)! - n \big )a_n \bigg ).
\]
\end{theorem}

\begin{proof}

The extreme points of $C_n^1$ consist of $2^n$ points in $\mathbb{R}^{n+1}$.  We have $2^{n-1}$ points of the form:
\[
(x_1 \,, \dots , x_{n-1} \,, a_n \,, y )^\top ,
\]
and $2^{n-1}$ points of the form:
\[
(x_1 \,, \dots , x_{n-1} \,, b_n \,, y )^\top .
\]
Moreover, each set of $2^{n-1}$ points contains exactly one point where $y\neq0$.

We first consider the $2^n-2$ points that lie on the hyperplane $y = 0$, and note that half lie on the hyperplane $x_n = a_n$ and half on the hyperplane $x_n = b_n$\,.

Consider the convex hull of the $2^{n-1}-1$ points that lie on the hyperplanes $y=0$ and $x_n = a_n$ projected onto the first $n-1$ components, and denote this  set by $B$. $B$ is a box of dimension $n-1$ with one corner ``cut off''.  Specifically, it is the convex hull of all points where each component is set to either its lower or upper bound, excepting the point $(b_1,\dots,b_n)^{\top}$.  
Therefore the $(n-1)$-volume is given by:

\begin{align*}
 \vol{n\!-\!1}{B}&= \prod_{i\in N_n} b_i - \left ( \frac{1}{(n-1)!}\left|\det
\begin{pmatrix}
0 & 0 & \cdots & 0& -b_1 \\
0 & 0 & \cdots &-b_2 & 0 \\
\vdots & \vdots & \iddots & \vdots & \vdots \\
0 & -b_{n-2} & \cdots & 0 & 0 \\
-b_{n-1} & 0 & \cdots & 0 & 0
\end{pmatrix}\right|\right )\\
 &= \prod_{i\in N_n} b_i -  \frac{1}{(n-1)!} \cdot \left| \prod_{i\in N_n} b_i \right| \\
 &= \frac{(n-1)!-1}{(n-1)!} \cdot \prod_{i\in N_n} b_i\,.
\end{align*}

Now consider the convex hull of the $2^{n-1}-1$ points that lie on the hyperplanes $y=0$ and $x_n = b_n$ projected onto the first $n-1$ components.  The convex hull of these points is also the set $B$.  Therefore, the $2^n-2$ points that lie on the hyperplane $y = 0$ form a prism with base $B$ in the parallel hyperplanes $x_n = a_n$ and $x_n = b_n$\,.  The $n$-volume of this prism, which we denote $P_n:=P_n(\bm{a},\bm{b})$, is:
\begin{align*}
 \vol{n}{P_n} = {\frac{(b_n-a_n)\left ((n-1)!-1\right )}{(n-1)!} \prod_{i\in N_n} b_i}\,.
\end{align*}

We observe that $C_n^1$ is the convex  hull of the $n$-dimensional prism, $P_n$\,, sitting in the $y = 0$ plane with two additional extreme points (that do not have their $y$-component equal to zero).
Let
\begin{align*}
 {
v_1 := \begin{pmatrix} b_1 \\ \vdots \\ b_{n-1} \\ a_n \\ b_1\ldots b_{n-1}a_n \end{pmatrix}
\quad \text{and} \quad
v_2:= \begin{pmatrix} b_1 \\ \vdots \\ b_{n-1} \\ b_n \\ b_1\ldots b_{n-1}b_n \end{pmatrix}.
}
\end{align*}
 We will calculate the volume of $C_n^1$ by first considering
 $$Q_{n+1}:= \conv(P_n \cup \{v_1\}), $$
 and then
$$C_n^1 = \conv(Q_{n+1} \cup \{v_2\}).$$

To obtain the volume of $Q_{n+1}$\,, we observe that it is a pyramid with base $P_n$ and apex $v_1$\,.  Therefore its volume is given by
\begin{align*}
\vol{n\!+\!1}{Q_{n+1}} = \frac{\text{base} \times \text{height}}{\text{dimension}}  &= \frac{(b_n-a_n)\frac{(n-1)!-1}{(n-1)!}\left(\prod_{i\in N_n} b_i \right)\left( a_n\prod_{i\in N_n} b_i\right)}{n+1}\\
&= \frac{(n-1)!-1}{(n+1)(n-1)!} \cdot a_n(b_n-a_n) \prod_{i\in N_n} b_i^2.
\end{align*}

We are now ready to compute the volume of
$C_n^1 = \conv(Q_{n+1} \cup \{v_2\})$.
Our first step is to compute the facets of $Q_{n+1}$\,. Note that the facets of a pyramid are exactly the base, in addition to the convex hull of the apex with each facet of the base. Therefore, we take each facet of $P_n$ and lift to find the equation of the hyperplane defined by each facet of the prism and $v_1$\,.

Observe $P_n$ has the $2n + 1$ facet defining inequalities given below:
\begin{align}
x_i &\geq 0, & i\in N_n\,; \label{eq:Pnfacet1} \\
b_i - x_i &\geq 0, & i\in N_n\,; \label{eq:Pnfacet2}\\
x_n - a_n &\geq 0; &\label{eq:Pnfacet3}\\
b_n - x_n &\geq 0; &\label{eq:Pnfacet4}\\
(n-2) \; - \;\sum_{i\in N_n}\frac{1}{b_i}x_i &\geq 0. &\label{eq:Pnfacet5}
\end{align}

By lifting, we obtain the following:
\begin{align}
x_i - \frac{1}{a_n \prod_{j\in N_{in}} b_j} y &\geq 0, & i\in N_n\,;  \tag{\ref{eq:Pnfacet1}*}  \label{eq:Pnfacet1s}\\
b_i - x_i &\geq 0, & i\in N_n\,; \tag{\ref{eq:Pnfacet2}*} \label{eq:Pnfacet2s}\\
x_n - a_n &\geq 0; & \tag{\ref{eq:Pnfacet3}*} \label{eq:Pnfacet3s}\\
b_n - \left (x_n + \frac{b_n-a_n}{a_n \prod_{j\in N_n} b_j} y \right ) &\geq 0;& \tag{\ref{eq:Pnfacet4}*} \label{eq:Pnfacet4s}\\
(n-2) - \left (\sum_{i\in N_n} \frac{1}{b_i} x_i - \frac{1}{a_n\prod_{j\in N_n} b_j} y\right ) &\geq 0. & \tag{\ref{eq:Pnfacet5}*}\label{eq:Pnfacet5s}
\end{align}
\addtocounter{equation}{1}
\eqref{eq:Pnfacet1s} -- \eqref{eq:Pnfacet5s}, along with
 \begin{align}
 y &\geq 0 \tag{\theequation$^*$} \label{eq:Pnfacet6s},
 \end{align}
give the facet description of $Q_{n+1}$\,.

We now calculate the ``extra" volume generated by including $v_2$\,.  This can be calculated by summing the volumes of
$$\conv (F \cup \{v_2\}), $$
where $F$ is a facet of $Q_{n+1}$ violated by the point $v_2$\,.

We substitute $v_2$ into each of the facet defining inequalities \eqref{eq:Pnfacet1s} -- \eqref{eq:Pnfacet6s}, noting that \eqref{eq:Pnfacet1s} and \eqref{eq:Pnfacet2s} each represent $n-1$ inequalities,  to calculate which are violated, i.e., separate $v_2$ from $Q_{n+1}$\,.
\begin{table}
\renewcommand{\arraystretch}{2.2}
 \begin{tabular}{|c|c|c|}
\hline
Equation(s) & Quantity when substituting $v_2$ into LHS & Inequality satisfied by $v_2$? \\\hline
\eqref{eq:Pnfacet1s} & $b_i\left(1-\frac{b_n}{a_n}\right) $ \; for $i \in N_n$ & \xmark \quad separates \\ \hline
\eqref{eq:Pnfacet2s} &$0$ \; for $i \in N_n$  & \cmark \quad satisfies\\\hline
\eqref{eq:Pnfacet3s} & $b_n - a_n$ & \cmark \quad satisfies \\\hline
\eqref{eq:Pnfacet4s} & $-(b_n-a_n) \cdot \frac{b_n}{a_n}$ & \xmark \quad separates \\\hline
\eqref{eq:Pnfacet5s} &$\frac{b_n}{a_n} - 1$ & \cmark \quad satisfies \\\hline
\eqref{eq:Pnfacet6s} & $\prod_{i\in N} b_i$ & \cmark \quad satisfies \\\hline
\end{tabular}
\caption{Substituting $v_2$ into the facet-describing inequalities of $Q_{n+1}$\,.} 
\label{table:subs}
\end{table}
In Table \ref{table:subs}, we see that $n$ facets separate $v_2$ from $Q_{n+1}$\,.  Note that $n-1$ of these have the same structure.

For each violated facet, we now calculate the $(n+1)$-volume of the convex hull of that facet with $v_2$\,. To do this, given a violated facet, we first calculate which extreme points of $Q_{n+1}$ lie on the facet. We then compute the $n$-volume of the facet, before computing the $(n+1)$-volume of the facet with $v_2$\,.

Let $F_i$ be the violated facet defined by \eqref{eq:Pnfacet1s} for some fixed $i \in N_n\,$.  Observe that an extreme point of $Q_{n+1}$\,, $\begin{pmatrix} \bm{x} \\ y \end{pmatrix} \in \R^{n+1}$, lies on this facet if and only if 
\begin{equation}
x_i = \frac{1}{a_n\prod_{j\in N_{in}}b_j} y.\tag{\ref{eq:Pnfacet1s}$_i$}\label{Pnfacet1si}
\end{equation}
Moreover, note that for all but one extreme point of $Q_{n+1}$\,, we have $y=0$.  Hence, the point lies on $F_i$ if and only if $x_i = 0$.  By substituting the final extreme point, $v_1$, into \eqref{Pnfacet1si}, we see that $v_1$ also lies on $F_i$\,.

 We wish to calculate the $(n+1)$-volume of $$\conv(F_i \cup \{v_2\}),$$
but first we compute the $n$-volume of $F_i$\,.

First, note that $v_1$ is the only extreme point on $F_i$ with $x_i \neq 0$, i.e., the remaining extreme points lie in the hyperplane $x_i = 0$.  Moreover, we have already seen that these points also lie in the hyperplane $y = 0$. Therefore, $F_i$ itself is a pyramid with $v_1$ being the apex, and the convex hull of the remaining points the base.

Define $\tilde{F}_i$ to be the base of the pyramid $F_i$\,.  The object $\tilde{F}_i$ is $(n-1)$-dimensional; more specifically, we observe that $\tilde{F}_i$ is a prism.

Consider the $2^{n-2}$ extreme points of the form,
\[
(
x_1 \,,
\ldots ,
x_{i-1} \,,
0 ,
x_{i+1} \,,
\ldots ,
x_{n-1} \,,
a_n \,,
0
)^\top. 
\]
The convex hull of these points projected onto the $x_i=0$, $x_n = a_n$\,, and $y=0$ planes is the box, 
\[
\{x \in \R^{n-2} ~:~ 0 \leq x_j \leq b_j\,, \;\; j\in N_{in}\},
\]
with volume 
$
\prod_{j\in N_{in}} b_j\,.
$
The remaining $2^{n-2}$ extreme points of $\tilde{F}_i$ have the form,
\[
(
x_1 \,,
\ldots 
,
x_{i-1} \,,
0 ,
x_{i+1} \,,
\dots ,
x_{n-1} \,,
b_n \,,
0
)^\top ,
\]
and the convex hull of these points projected onto the $x_i=0$, $x_n = b_n$\,, and $y=0$ planes is an identical box.  Therefore, $\tilde{F}_i$ is a prism with $(n-1)$-volume
\[
\vol{n\!-\!1}{\tilde{F}_i} = (b_n - a_n) \prod_{j\in N_{in}} b_j\,. 
\]

We now compute the $n$-volume of $F_i = \conv(\tilde{F}_i \cup \{v_1\})$:
\begin{align*}
\vol{n}{F_i} &= \frac{\text{base} \times \text{height}}{\text{dimension}}  \\
&= \frac{(b_n - a_n) 
\prod_{j\in N_{in}}
b_j \cdot \sqrt{b_i^2 + a_n^2 \prod_{j\in N_n} b_j^2}}{n} \\
&= \frac{(b_n - a_n) 
\prod_{j\in N_{in}}
b_j \cdot \sqrt{b_i^2 \left(1 + {a_n^2 
\prod_{j\in N_{in}}
b_j^2}\right)}}{n} \\
&= \frac{(b_n - a_n) \cdot \prod_{j\in N_n} b_j \cdot \sqrt{1 + {a_n^2 
\prod_{j\in N_{in}} 
b_j^2}}}{n}.
\end{align*}

We now have the volume of the base of 
$\conv(F_i \cup \{v_2\}),$
and calculate the height of the apex, i.e., $v_2$\,, above this base, $F_i$\,.  This is the distance from the point $v_2$ to the hyperplane \begin{equation*}
x_i = \frac{1}{a_n\prod_{j\in N_{in}} b_j} y,
\end{equation*}
and is calculated as follows:

\begin{align*}
\frac{\left| b_i - \frac{
\prod_{j\in N} 
b_j}{a_n
\prod_{j\in N_{in}}
b_j} \right|}{\sqrt{1^2 + \frac{1}{a_n
\prod_{j\in N_{in}}
b_j^2}}} 
&= \frac{-b_i + b_i\cdot \frac{b_n}{a_n}}{\sqrt{\frac{a_n^2 
\prod_{j\in N_{in}}
b_j^2 + 1}{a_n^2 
\prod_{j\in N_{in}}
b_j^2}}} 
&= \frac{a_n 
\left(\prod_{j\in N_{in}}
b_j\right) \left (b_i\left(\frac{b_n}{a_n} - 1\right) \right )}{\sqrt{a_n^2
\prod_{j\in N_{in}}
b_j^2 + 1}}
&=\frac{(b_n-a_n) \cdot 
\prod_{j\in N_n} 
b_j}{\sqrt{a_n^2 
\prod_{j\in N_{in}}
b_j^2 + 1}}.
\end{align*}
Using this, we calculate the $(n+1)$-volume of $\text{conv}(F_i \cup \{v_2\})$ to be
\begin{equation*}
\frac{(b_n - a_n)^2 \cdot \prod_{j\in N_n} b_j^2}{n(n+1)}.
\end{equation*}

Note that this volume does {\bf not} depend on $i$.  Therefore, each of the $n-1$ facets defined by \eqref{eq:Pnfacet1s} will produce the same quantity.  In other words, 
$
\vol{n\!+\!1}{\conv (F_i \cup \{v_2\})},
$
is the same for $i \in N_n$\,, and 
\[
\sum_{i\in N_n}\vol{n\!+\!1}{\conv (F_i \cup \{v_2\})} = \frac{(n-1)(b_n - a_n)^2 \cdot \prod_{j\in N_n} b_j^2}{n(n+1)}.
\]

We now consider the final violated facet, denoted $F$, defined by the inequality
\begin{equation}
b_n - \left (x_n + \frac{b_n-a_n}{a_n \prod_{j\in N_n} b_j} y \right ) \geq 0 \tag{\ref{eq:Pnfacet4s}}.
\end{equation}
Recall that for all but one extreme point of $Q_{n+1}$, we have y = 0. Points of this form lie on $F$ if and only if $x_n = b_n$\,.  Moreover, by substituting the final
extreme point, $v_1$\,, into \eqref{eq:Pnfacet4s}, we see that $v_1$ also lies on $F$.
 
We wish to calculate the $(n+1)$-volume of $\conv(F \cup \{v_2\}),$
but first we compute the $n$-volume of $F$.

Note that $v_1$ is the only extreme point on $F$ with $x_n \neq b_n$\,, i.e., the remaining extreme points lie in the hyperplane $x_n = b_n$\,.  Moreover, we have already seen that these points also lie in the hyperplane $y = 0$. Therefore, similarly to before, $F$ itself is a pyramid with $v_1$ being the apex, and the convex hull of the remaining points the base.

Define $\tilde{F}$ to be the base of pyramid $F$.
$\tilde{F}$ is $(n-1)$-dimensional; more specifically, we observe that it is a box with one corner “cut off”.  In particular, $\tilde{F}$ is the convex hull of extreme points with the form,
$
(
x_1 \,,
\ldots ,
x_{n-1} \,,
b_n \,,
0
)^\top,
$
that have {\it at least} one $x_i = 0$ for $i \in N_n$\,.
Hence, the $(n-1)$-volume of $\tilde{F}$ is: \begin{equation*}\vol{n\!-\!1}{\tilde{F}} = \frac{(n-1)!-1}{(n-1)!} \cdot \prod_{i\in N_n} b_i\,. \end{equation*}

We now compute the $n$-volume of $F = \conv(\tilde{F} \cup \{v_1\})$:
\begin{align*}
\vol{n}{F} &= \frac{\text{base} \times \text{height}}{\text{dimension}}  \\
&= \frac{\frac{(n-1)! - 1}{(n-1)!} 
\prod_{j\in N_{n}}
b_j \cdot \sqrt{(b_n-a_n)^2 + a_n^2 \prod_{j\in N_n} b_j^2}}{n} \\
&= \frac{(n-1)! - 1}{n!} 
\prod_{j\in N_{n}}
b_j \cdot \sqrt{(b_n-a_n)^2 + a_n^2 \prod_{j\in N_n} b_j^2}\,.
\end{align*}

We now have the volume of the base of 
$\conv(F \cup \{v_2\}),$
and calculate the height of the apex, i.e, $v_2$\,, above this base, $F$. This is the distance from the point $v_2$ to the hyperplane,
\[
x_n + \frac{b_n-a_n}{a_n \prod_{j\in N_n}b_j } y = b_n\,,
\]
and is calculated as follows:

\begin{align*}
&\frac{\left| b_n + \frac{(b_n-a_n)}{a_n
\prod_{j\in N_n}
b_j} \cdot 
\prod_{j\in N} 
b_j - b_n \right|}{\sqrt{1 + \left (\frac{b_n-a_n}{a_n
\prod_{j\in N_n}
b_j} \right )^2}}
= \frac{(b_n-a_n) \cdot \frac{b_n}{a_n}}{\sqrt{\frac{a_n^2 
\prod_{j\in N_n} 
b_j^2 + (b_n-a_n)^2}{a_n^2 
\prod_{j\in N_n} 
b_j^2}}} \\
&\qquad= \frac{a_n 
\left(\prod_{j\in N_n} 
b_j\right) \cdot (b_n-a_n) \cdot \frac{b_n}{a_n} }{\sqrt{a_n^2 
\prod_{j\in N_n} 
b_j^2 + (b_n-a_n)^2}}
= \frac{(b_n-a_n) \cdot 
\prod_{j\in N}
b_j }{\sqrt{a_n^2 
\prod_{j\in N_n} 
b_j^2 + (b_n-a_n)^2}}.
\end{align*}
Using this distance, we calculate the $(n+1)$-volume of $\text{conv}(F \cup \{v_2\})$ to be
\begin{equation*}
\frac{(n-1)! - 1}{n!(n+1)} \cdot b_n(b_n-a_n) \cdot \prod_{j\in N_n} b_j^2\,.
\end{equation*}
Finally we sum the appropriate quantities to calculate the volume of $C_n^1$\,:

\begin{align*}
&\vol{n\!+\!1}{C_n^1} = \vol{n\!+\!1}{Q_{n+1}} + 
\sum_{i\in N_n}
\vol{n\!+\!1}{\conv (F_i \cup \{v_2\})} + \vol{n\!+\!1}{\conv (F \cup \{v_2\})} \\
&\quad= \frac{(n-1)! - 1}{(n+1)(n-1)!} \cdot a_n(b_n-a_n) \cdot 
\prod_{i\in N_n} 
b_i^2 + \frac{(n-1)(b_n-a_n)^2 \cdot 
\prod_{j\in N_n} 
b_j^2}{n(n+1)}  \\
&\qquad \qquad \qquad \qquad \qquad \qquad + \frac{(n-1)! - 1}{(n+1)n(n-1)!} \cdot b_n(b_n-a_n) \cdot 
\prod_{j\in N_n}
b_j^2 \\
&\quad= \frac{(b_n-a_n) \cdot 
\prod_{j\in N_n} 
b_j^2 }{(n+1)!} \cdot \bigg ( n\left((n-1)! - 1\right)a_n + (n-1)(n-1)!(b_n-a_n) + b_n\left((n-1)! - 1\right) \bigg)\\
&\quad= \frac{(b_n-a_n) 
\prod_{i\in N_n}
b_i^2}{(n+1)!} \;\bigg ( (n! - 1)b_n + \big ((n-1)! - n \big )a_n \bigg ).
\end{align*}

\end{proof}


\section{Outlook}\label{sec:outlook}

An interesting next step would be 
to give a complete inequality description and volume formula 
for 
$C^2_n$\,, which would generalize what we presented above
for $C^1_n$ 
and would also include the tetrahedron $C^2_2$\,.
Of course, it would be nice to treat 
 $C^3_n$ as well, 
but that would carry all the complexity of
$C^3_3$\,, which is already quite involved
(see \cite{SpeakmanLee2015,SpeakAkerov2019}).
Other directions are to generalize what we did for 
 $C^1_n$ to mixed-sign domains (see \cite{MakhoulSpeakman2025} for 
 the general case of $n=3$), and to analyze
 ``double-McCormick relaxations'' of $C^1_n$ 
 and compare them (via volume) with $C^1_n$
 (see \cite{SpeakmanLee2015} for the case of $C^3_3$).
 

\backmatter

\bmhead{Acknowledgements}
This work was begun at the Workshop on Circuit Diameters and Augmentation: Recent Advances in Linear and Integer Optimization; sponsored under NSF grant 2006183 (PI: S. Borgwardt).  J. Lee was supported in part by U.S.A. Office of Naval Research grant N00014-24-1-2694. Some use was made of a Claude LLM to refine and streamline the presentation of our proof of Theorem \ref{thm:facets}.

\bmhead{Data Availability Statement}
Data sharing not applicable to this article as no datasets were generated or analyzed during the current study.

\bmhead{Disclaimer} The views expressed in this article are those of the authors and do not reflect the official policy or position of the U.S. Naval Academy, Department of the Navy, the Department of Defense, or the U.S. Government.

\bibliography{mv-journal} 

@article{LSS_STOGO,
title={On the convex hull of the graph of a simple monomial, {STOGO 2025, Proceedings of the XVI Workshop on Global Optimization}},
author="Lee, Jon
and Skipper, Daphne
and Speakman, Emily",
year={2025},
editors={Johanna Skåntorp and Jan Kronqvist},
pages={167--171},
}

@article{SpeakAkerov2019,
title = "Computing the volume of the convex hull of the graph of a trilinear monomial using mixed volumes",
journal = "Discrete Appllied Mathematics",
year = "2022",
volume="308",
pages="36--45",
author = {Emily Speakman and Gennadiy Averkov},
note={\url{https://doi.org/10.1016/j.dam.2019.09.007}},
}

@PHDTHESIS{SpeakmanThesis,
  author =       {Emily Speakman},
  title =        {Volumetric Guidance for Handling Triple Products in Spatial Branch-and-Bound},
  school =       {University of Michigan},
  year =         {2017},
  type =         {{Ph.D. Dissertation}},
  month =        {April},
note={\url{https://deepblue.lib.umich.edu/items/17bbe5ef-d264-4575-beb9-9581d4b239f7}},
}

@article {SpeakmanLee2015,
    AUTHOR = {Speakman, Emily and Lee, Jon},
     TITLE = {Quantifying double {M}c{C}ormick},
  JOURNAL = {Mathematics of Operations Research},
    VOLUME = {42},
      YEAR = {2017},
    NUMBER = {4},
     PAGES = {1230--1253},
note={\url{https://doi.org/10.1287/moor.2017.0846}},
}

@article{LM1994,
    AUTHOR = {Lee, Jon and Morris, Jr., Walter D.},
     TITLE = {Geometric comparison of combinatorial polytopes},
JOURNAL = {Discrete Applied Mathematics},
    VOLUME = {55},
      YEAR = {1994},
     PAGES = {163--182},
     note={\url{https://doi.org/10.1016/0166-218X(94)90006-X}},
}

@article{KLS1197,
title = {The volume of relaxed Boolean-quadric and cut polytopes},
journal = {Discrete Mathematics},
volume = {163},
number = {1},
pages = {293--298},
year = {1997},
author = {Chun-Wa Ko and Jon Lee and Einar Steingrímsson},
note={\url{https://doi.org/10.1016/0012-365X(95)00343-U}},
}

@book{zieg,
  address = {New York},
  author = {Ziegler, Günter M.},
  publisher = {Springer-Verlag},
  title = {Lectures on Polytopes. Graduate Texts in Mathematics, 152},
  note={\url{https://doi.org/10.1007/978-1-4613-8431-1}},
  year = 1995,
}

@article{Steingrimsson1994,
author = {Steingrímsson, E.},
journal = {Discrete \& Computational Geometry},
number = {4},
pages = {465--479},
title = {A decomposition of 2-weak vertex-packing polytopes},
note={\url{https://doi.org/10.1007/BF02574393}},
volume = {12},
year = {1994},
}

@misc{MakhoulSpeakman2025,
      title={Volume formulae for the convex hull of the graph of a trilinear monomial: A complete characterization for general box domains}, 
      author={Lillian Makhoul and Emily Speakman},
      year={2025},
note={\url{https://arxiv.org/abs/2512.13964}}, 
}

@article{quad,
author = {Cafieri, Sonia and Lee, Jon and Liberti, Leo},
title = {On convex relaxations of quadrilinear terms},
year = {2010},
volume = {47},
number = {4},
note= {\url{https://doi.org/10.1007/s10898-009-9484-1}},
journal = {Journal of Global Optimization},
pages = {661--685},
}

@article{SMITH1999457,
title = {A symbolic reformulation/spatial branch-and-bound algorithm for the global optimisation of nonconvex MINLPs},
journal = {Computers \& Chemical Engineering},
volume = {23},
number = {4},
pages = {457--478},
year = {1999},
note={\url{https://doi.org/10.1016/S0098-1354(98)00286-5}},
author = {E.M.B. Smith and C.C. Pantelides},
}

@Article{LeeSkipperSpeakmanMPB2018,
author="Lee, Jon
and Skipper, Daphne
and Speakman, Emily",
 volume = {170},
 pages = {121--140},
title="Algorithmic and modeling insights via volumetric comparison of polyhedral relaxations",
journal="Mathematical Programming",
year="2018",
note="\url{https://doi.org/10.1007/s10107-018-1272-6}",
}

@Incollection{SpeakmanYuLee,
author="Speakman, Emily and Yu, Han and Lee, Jon",
editor="Salvagnin, D. and Lombardi, M.",
title="Experimental validation of volume-based comparison for double-{M}c{C}ormick relaxations",
booktitle="CPAIOR 2017",
year="2017",
publisher="Springer",
address="Cham",
pages="229--243",
note={\url{https://doi.org/10.1007/978-3-319-59776-8_19}},
}

@Article{SpeakmanLee_Branching,
author="Speakman, Emily and Lee, Jon",
title="On branching-point selection for trilinear monomials in spatial branch-and-bound: the hull relaxation",
journal="Journal of Global Optimization",
volume = {72},
 pages = {129--153},
year="2018",
note = {\url{https://doi.org/10.1007/s10898-018-0620-7}},
}

@article{LeeSkipper2017,
 author    = {Lee, Jon and Skipper, Daphne},
 title     = "Volume computation for sparse boolean quadric relaxations",
JOURNAL = {Discrete Applied Mathematics},
volume = {275},
pages = {79--94},
year = {2020},
note={\url{https://doi.org/10.1016/j.dam.2018.10.038}},
}

@article{McCormick76,
  title={Computability of global solutions to factorable nonconvex programs: Part {I} -— Convex underestimating problems},
  author={McCormick, Garth P},
  journal={Mathematical Programming, Series B},
  volume={10},
  number={1},
  pages={147--175},
  year={1976},
  publisher={Springer},
note = {\url{https://doi.org/10.1007/BF01580665}},

}

@article{rikun1997,
  title={A convex envelope formula for multilinear functions},
  author={Rikun, Anatoliy D.},
  journal={Journal of Global Optimization},
  volume={10},
  pages={425--437},
  year={1997},
note = {\url{https://doi.org/10.1023/A:1008217604285}}
}

@article{meyer2004trilinear,
  title={Trilinear monomials with mixed sign domains: Facets of the convex and concave envelopes},
  author={Meyer, Clifford and Floudas, Christodoulos},
  journal={Journal of Global Optimization},
  volume={29},
  pages={125--155},
  year={2004},
  publisher={Springer},
note = {\url{https://doi.org/10.1023/B:JOGO.0000042112.72379.e6}}
}

@article{meyer04mixed,
  title={Trilinear monomials with mixed sign domains: Facets of the convex and concave envelopes},
  author={Meyer, Clifford A. and Floudas, Christodoulos A.},
  journal={Journal of Global Optimization},
  volume={29},
  pages={125--155},
  year={2004},
  publisher={Springer},
  doi = {10.1023/B:JOGO.0000042112.72379.e6}
}

\end{document}